\documentclass[11pt]{article}
\usepackage{amsthm,geometry,amssymb,amsmath,cite,setspace,hyperref,float,enumerate,algorithm2e}
\newtheorem{theorem}{Theorem}
\newtheorem{lemma}[theorem]{Lemma}

\newcommand{\kr}{\mbox{\tiny $\bigcirc$}}

\allowdisplaybreaks
\usepackage{lineno}

\title{Reconfiguring dominating sets in minor-closed graph classes}
\author{Dieter Rautenbach and Johannes Redl\\[3mm]
\normalsize Institute of Optimization and Operations Research, Ulm University, Germany\\
\texttt{$\{$dieter.rautenbach,johannes.redl$\}$@uni-ulm.de}}
\date{}

\begin{document}
\maketitle
\onehalfspace

\begin{abstract}
For a graph $G$, two dominating sets $D$ and $D'$ in $G$,
and a non-negative integer $k$,
the set $D$ is said to {\it $k$-transform} to $D'$
if there is a sequence $D_0,\ldots,D_\ell$ of dominating sets in $G$ such that 
$D=D_0$, 
$D'=D_\ell$, 
$|D_i|\leq k$ for every $i\in \{ 0,1,\ldots,\ell\}$,
and $D_i$ arises from $D_{i-1}$ 
by adding or removing one vertex
for every $i\in \{ 1,\ldots,\ell\}$.
We prove that there 
is some positive constant $c$
and there are toroidal graphs $G$ 
of arbitrarily large order $n$,
and two minimum dominating sets $D$ and $D'$ in $G$
such that $D$ $k$-transforms to $D'$
only if $k\geq \max\{ |D|,|D'|\}+c\sqrt{n}$.
Conversely, 
for every hereditary class ${\cal G}$ 
that has balanced separators of order $n\mapsto n^\alpha$
for some $\alpha<1$,
we prove that there is some positive constant $C$ such that,
if $G$ is a graph in ${\cal G}$ of order $n$, 
and $D$ and $D'$ are two dominating sets in $G$,
then 
$D$ $k$-transforms to $D'$
for $k=\max\{ |D|,|D'|\}+\lfloor Cn^\alpha\rfloor$.\\[2mm]
\end{abstract}

\section{Introduction}

We consider finite, simple, and undirected graphs, and use standard notation and terminology.
Let $G$ be a graph. 
A set $D$ of vertices of $G$ is a {\it dominating set in} $G$ 
if every vertex of $G$ belongs to $D$ or has a neighbor in $D$.
Let $D$ and $D'$ be two dominating sets in $G$.
The dominating sets $D$ and $D'$ are {\it adjacent} 
if $|D\setminus D'|+|D'\setminus D|=1$,
that is, if $D'$ arises from $D$ by adding or removing one vertex.
Let $k$ be a positive integer. 
We say that $D$ {\it $k$-transforms to} $D'$, and write 
$D\stackrel{k}{\longleftrightarrow}D'$,
if there is a sequence $D_0,\ldots,D_\ell$ of dominating sets in $G$ such that 
$D=D_0$, 
$D'=D_\ell$, 
$|D_i|\leq k$ for every $i\in \{ 0,1,\ldots,\ell\}$,
and $D_{i-1}$ is adjacent to $D_i$ for every $i\in \{ 1,\ldots,\ell\}$.
Let $D_k(G)$ be the graph whose vertices are 
the dominating sets in $G$ that are of order at most $k$,
and whose edges are defined by the above adjacency between dominating sets.

The structure and, in particular, reachability, connectivity, and distance problems in $D_k(G)$ 
have been studied in \cite{alfakl,hase,sumoni}
and --- with a focus on algorithmic and complexity results ---
in \cite{haitnionsute,lomoparasa}.
In \cite{ni} a general survey on reconfiguration problems for several types of sets of vertices in a graph is given.
If $\gamma(G)$ and $\Gamma(G)$ denote the minimum and maximum order of dominating sets in $G$
that are minimal with respect to inclusion,
respectively, then it is easy to see that $D_{\Gamma(G)+\gamma(G)}(G)$ is always connected.
Answering a question of Haas and Seyffarth \cite{hase}, 
Suzuki et al.~\cite{sumoni} construct planar graphs $G$
for which $D_{\Gamma(G)+1}(G)$ is disconnected.
As they point out, it is unknown whether $D_{\Gamma(G)+2}(G)$ is connected for every graph $G$.

Inspired by the cited research and this open problem,
we take a slightly different point of view, 
considering 
--- roughly speaking --- how many additional vertices must
be allowed in order to transform one given dominating set to another given one.
More precisely, 
if $D$ and $D'$ are dominating sets in a graph $G$,
then let $\partial\gamma_G(D,D')$ equal $k-\max\{ |D|,|D'|\}$,
where $k$ is the smallest positive integer with $D\stackrel{k}{\longleftrightarrow}D'$.
In particular, 
$$D\longleftarrow\hspace{-0.2em}\stackrel{\max\{ |D|,|D'|\}+\partial\gamma_G(D,D')}{-\hspace{-0.4em}-\hspace{-0.4em}-\hspace{-0.4em}-\hspace{-0.4em}-\hspace{-0.4em}-\hspace{-0.4em}-\hspace{-0.4em}-\hspace{-0.4em}-\hspace{-0.4em}-\hspace{-0.4em}-\hspace{-0.4em}-\hspace{-0.4em}-\hspace{-0.4em}-\hspace{-0.4em}-\hspace{-0.4em}-}\hspace{-0.2em}\longrightarrow D',$$
that is, allowing $\partial\gamma_G(D,D')$ more vertices than 
contained in the larger of the two dominating sets,
one can transform $D$ to $D'$.
It is easy to see that 
$$\partial\gamma_G(D,D')\leq \min\left\{ \gamma(G),\frac{n(G)}{2}\right\}$$
for every graph $G$ of order $n(G)$.
It seems an interesting problem to determine the best possible upper bound on $\partial\gamma_G(D,D')$
in terms of the order $n(G)$;
for general graphs as well as for graphs from restricted graph classes.
Known results \cite{haitnionsute} imply, for instance, 
that $\partial\gamma_G(D,D')\leq 1$
whenever $G$ is a cograph, a forest, or an interval graph.
The results of the present paper were obtained 
wondering whether $\partial\gamma_G(D,D')$
can be upper bounded in terms of the maximum degree $\Delta(G)$ of $G$.

Our first result shows that this is not possible.

\begin{theorem}\label{theorem1}
There is a positive constant $c$ such that, 
for every positive integer $n$, 
there is a $4$-regular graph $G$ of order at least $n$
that can be embedded on the torus, 
and there are two dominating sets $D$ and $D'$ of $G$, 
both of order $n(G)/5$,
such that 
$$\partial\gamma_G(D,D')\geq c\sqrt{n(G)}.$$
\end{theorem}
Our second result shows that the lower bound in Theorem \ref{theorem1}
has the right order of magnitude.
Rather than considering only graphs embedded on the torus, 
we consider graphs with sublinear balanced separators.
In order to phrase our second result, we need some more definitions:
A graph $G$ has a {\it balanced separator of order $k$} 
if there is a set $S$ of at most $k$ vertices of $G$ 
as well as a partition of the vertex set $V(G)$ of $G$ into three sets
$S$, $A$, and $B$
such that $|A|,|B|\leq 2n(G)/3$, and $G$ contains no edge between $A$ and $B$.
A hereditary class ${\cal G}$ of graphs has {\it balanced separators of order $n\mapsto n^\alpha$} 
if there is some positive constant $c$
such that every graph $G$ in ${\cal G}$
has a balanced separator of order $cn(G)^\alpha$.
It is known that minor-closed graph classes \cite{kare}
such as planar graphs \cite{lita},
toroidal graphs,
and graphs of bounded genus \cite{gihuta}
have balanced separators of order $n\mapsto \sqrt{n}$.

Here is our second main result.

\begin{theorem}\label{theorem2}
Let ${\cal G}$ be a hereditary class of graphs that has 
balanced separators of order $n\mapsto n^\alpha$
for some $\alpha<1$.
There is some positive constant $C$ such that,
if $G$ is a graph in ${\cal G}$, and $D$ and $D'$ are two dominating sets in $G$,
then 
$$\partial\gamma_G(D,D')\leq C n(G)^\alpha.$$
\end{theorem}
The proofs of our two results are given in the following two sections.

\section{Proof of Theorem \ref{theorem1}}

The $4$-regular graph $G$ that we construct for this result 
arises by applying suitable vertex identifications
to a sufficiently large subgraph of the infinite grid graph 
$\mathbb{Z}^2$ illustrated in Figure \ref{fig1}.
In this figure we also illustrate two dominating sets $D_\Box$ and $D_{\kr}$ of $\mathbb{Z}^2$,
one indicated by squares $\Box$ 
and the second indicated by circles ${\kr}$.
The graph $G$ will be constructed in such a way that 
$D=V(G)\cap D_\Box$
and 
$D'=V(G)\cap D_{\kr}$
are minimum dominating sets of $G$.
In fact, since $G$ is $4$-regular, 
every dominating set in $G$ contains at least $|D|=|D'|=n(G)/5$ vertices.

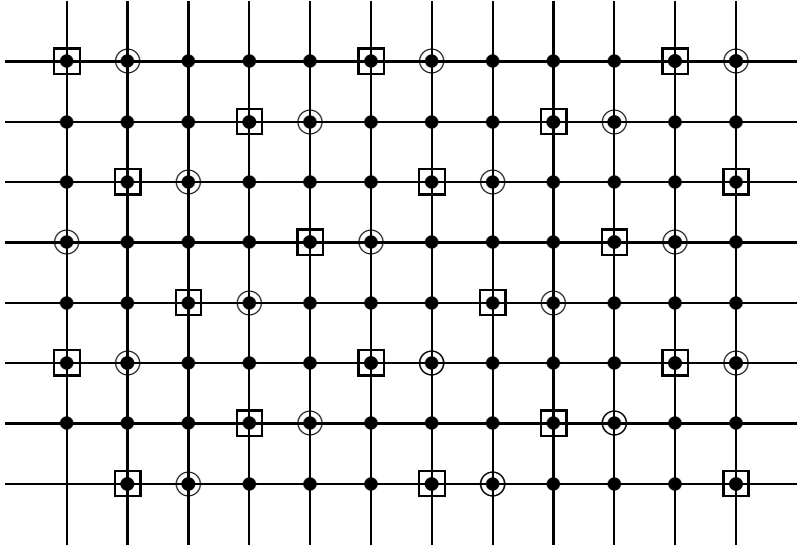
\begin{figure}[h]
\begin{center}
\unitlength 0.8mm 
\linethickness{0.4pt}
\ifx\plotpoint\undefined\newsavebox{\plotpoint}\fi 
\begin{picture}(130,90)(0,0)
\put(20,10){\circle*{2}}
\put(10,20){\circle*{2}}
\put(20,10){\circle*{2}}
\put(20,30){\circle*{2}}
\put(10,40){\circle*{2}}
\put(10,50){\circle*{2}}
\put(10,60){\circle*{2}}
\put(10,70){\circle*{2}}
\put(10,80){\circle*{2}}
\put(30,10){\circle*{2}}
\put(20,20){\circle*{2}}
\put(20,30){\circle*{2}}
\put(20,40){\circle*{2}}
\put(20,50){\circle*{2}}
\put(20,60){\circle*{2}}
\put(20,70){\circle*{2}}
\put(20,80){\circle*{2}}
\put(30,10){\circle*{2}}
\put(30,20){\circle*{2}}
\put(30,30){\circle*{2}}
\put(30,40){\circle*{2}}
\put(10,30){\circle*{2}}
\put(30,50){\circle*{2}}
\put(30,60){\circle*{2}}
\put(30,70){\circle*{2}}
\put(30,80){\circle*{2}}
\put(40,10){\circle*{2}}
\put(40,20){\circle*{2}}
\put(40,30){\circle*{2}}
\put(40,40){\circle*{2}}
\put(40,50){\circle*{2}}
\put(40,60){\circle*{2}}
\put(40,70){\circle*{2}}
\put(40,80){\circle*{2}}
\put(50,10){\circle*{2}}
\put(50,20){\circle*{2}}
\put(50,30){\circle*{2}}
\put(50,40){\circle*{2}}
\put(50,50){\circle*{2}}
\put(50,60){\circle*{2}}
\put(40,70){\circle*{2}}
\put(50,80){\circle*{2}}
\put(60,10){\circle*{2}}
\put(60,20){\circle*{2}}
\put(60,30){\circle*{2}}
\put(60,40){\circle*{2}}
\put(50,50){\circle*{2}}
\put(60,60){\circle*{2}}
\put(50,70){\circle*{2}}
\put(60,80){\circle*{2}}
\put(90,70){\circle*{2}}
\put(110,80){\circle*{2}}
\put(70,10){\circle*{2}}
\put(70,20){\circle*{2}}
\put(60,30){\circle*{2}}
\put(70,40){\circle*{2}}
\put(60,50){\circle*{2}}
\put(70,60){\circle*{2}}
\put(100,50){\circle*{2}}
\put(70,70){\circle*{2}}
\put(60,70){\circle*{2}}
\put(70,80){\circle*{2}}
\put(100,70){\circle*{2}}
\put(120,80){\circle*{2}}
\put(70,10){\circle*{2}}
\put(80,20){\circle*{2}}
\put(70,30){\circle*{2}}
\put(80,40){\circle*{2}}
\put(110,30){\circle*{2}}
\put(80,50){\circle*{2}}
\put(70,50){\circle*{2}}
\put(80,60){\circle*{2}}
\put(110,50){\circle*{2}}
\put(80,70){\circle*{2}}
\put(80,80){\circle*{2}}
\put(80,10){\circle*{2}}
\put(90,20){\circle*{2}}
\put(120,10){\circle*{2}}
\put(90,30){\circle*{2}}
\put(80,30){\circle*{2}}
\put(90,40){\circle*{2}}
\put(120,30){\circle*{2}}
\put(90,50){\circle*{2}}
\put(90,60){\circle*{2}}
\put(90,70){\circle*{2}}
\put(110,80){\circle*{2}}
\put(90,80){\circle*{2}}
\put(100,10){\circle*{2}}
\put(90,10){\circle*{2}}
\put(100,20){\circle*{2}}
\put(100,30){\circle*{2}}
\put(100,40){\circle*{2}}
\put(100,50){\circle*{2}}
\put(100,60){\circle*{2}}
\put(100,70){\circle*{2}}
\put(120,80){\circle*{2}}
\put(100,80){\circle*{2}}
\put(110,10){\circle*{2}}
\put(110,20){\circle*{2}}
\put(110,30){\circle*{2}}
\put(110,40){\circle*{2}}
\put(110,50){\circle*{2}}
\put(110,60){\circle*{2}}
\put(110,70){\circle*{2}}
\put(110,80){\circle*{2}}
\put(120,10){\circle*{2}}
\put(120,20){\circle*{2}}
\put(120,30){\circle*{2}}
\put(120,40){\circle*{2}}
\put(120,50){\circle*{2}}
\put(120,60){\circle*{2}}
\put(120,70){\circle*{2}}
\put(120,80){\circle*{2}}
\put(0,10){\line(1,0){10}}
\put(0,20){\line(1,0){10}}
\put(10,10){\line(1,0){10}}
\put(0,30){\line(1,0){10}}
\put(0,40){\line(1,0){10}}
\put(0,50){\line(1,0){10}}
\put(0,60){\line(1,0){10}}
\put(0,70){\line(1,0){10}}
\put(0,80){\line(1,0){10}}
\put(20,10){\line(1,0){10}}
\put(10,20){\line(1,0){10}}
\put(10,30){\line(1,0){10}}
\put(10,40){\line(1,0){10}}
\put(10,50){\line(1,0){10}}
\put(10,60){\line(1,0){10}}
\put(10,70){\line(1,0){10}}
\put(10,80){\line(1,0){10}}
\put(30,10){\line(1,0){10}}
\put(20,20){\line(1,0){10}}
\put(20,30){\line(1,0){10}}
\put(20,40){\line(1,0){10}}
\put(20,50){\line(1,0){10}}
\put(20,60){\line(1,0){10}}
\put(20,70){\line(1,0){10}}
\put(20,80){\line(1,0){10}}
\put(30,10){\line(1,0){10}}
\put(30,20){\line(1,0){10}}
\put(30,30){\line(1,0){10}}
\put(30,40){\line(1,0){10}}
\put(30,50){\line(1,0){10}}
\put(30,60){\line(1,0){10}}
\put(30,70){\line(1,0){10}}
\put(30,80){\line(1,0){10}}
\put(40,10){\line(1,0){10}}
\put(40,20){\line(1,0){10}}
\put(40,30){\line(1,0){10}}
\put(40,40){\line(1,0){10}}
\put(40,50){\line(1,0){10}}
\put(40,60){\line(1,0){10}}
\put(40,70){\line(1,0){10}}
\put(40,80){\line(1,0){10}}
\put(50,10){\line(1,0){10}}
\put(50,20){\line(1,0){10}}
\put(50,30){\line(1,0){10}}
\put(50,40){\line(1,0){10}}
\put(50,50){\line(1,0){10}}
\put(50,60){\line(1,0){10}}
\put(40,70){\line(1,0){10}}
\put(50,80){\line(1,0){10}}
\put(60,10){\line(1,0){10}}
\put(60,20){\line(1,0){10}}
\put(60,30){\line(1,0){10}}
\put(60,40){\line(1,0){10}}
\put(50,50){\line(1,0){10}}
\put(60,60){\line(1,0){10}}
\put(60,70){\line(1,0){10}}
\put(50,70){\line(1,0){10}}
\put(60,80){\line(1,0){10}}
\put(90,70){\line(1,0){10}}
\put(110,80){\line(1,0){10}}
\put(70,10){\line(1,0){10}}
\put(70,20){\line(1,0){10}}
\put(60,30){\line(1,0){10}}
\put(70,40){\line(1,0){10}}
\put(70,50){\line(1,0){10}}
\put(60,50){\line(1,0){10}}
\put(70,60){\line(1,0){10}}
\put(100,50){\line(1,0){10}}
\put(70,70){\line(1,0){10}}
\put(70,80){\line(1,0){10}}
\put(70,10){\line(1,0){10}}
\put(80,20){\line(1,0){10}}
\put(80,30){\line(1,0){10}}
\put(70,30){\line(1,0){10}}
\put(80,40){\line(1,0){10}}
\put(110,30){\line(1,0){10}}
\put(80,50){\line(1,0){10}}
\put(80,60){\line(1,0){10}}
\put(80,70){\line(1,0){10}}
\put(80,80){\line(1,0){10}}
\put(90,10){\line(1,0){10}}
\put(80,10){\line(1,0){10}}
\put(90,20){\line(1,0){10}}
\put(120,10){\line(1,0){10}}
\put(90,30){\line(1,0){10}}
\put(90,40){\line(1,0){10}}
\put(90,50){\line(1,0){10}}
\put(90,60){\line(1,0){10}}
\put(90,70){\line(1,0){10}}
\put(110,80){\line(1,0){10}}
\put(90,80){\line(1,0){10}}
\put(100,10){\line(1,0){10}}
\put(100,20){\line(1,0){10}}
\put(100,30){\line(1,0){10}}
\put(100,40){\line(1,0){10}}
\put(100,50){\line(1,0){10}}
\put(100,60){\line(1,0){10}}
\put(100,70){\line(1,0){10}}
\put(100,80){\line(1,0){10}}
\put(110,10){\line(1,0){10}}
\put(110,20){\line(1,0){10}}
\put(110,30){\line(1,0){10}}
\put(110,40){\line(1,0){10}}
\put(110,50){\line(1,0){10}}
\put(110,60){\line(1,0){10}}
\put(110,70){\line(1,0){10}}
\put(110,80){\line(1,0){10}}
\put(120,10){\line(1,0){10}}
\put(120,20){\line(1,0){10}}
\put(120,30){\line(1,0){10}}
\put(120,40){\line(1,0){10}}
\put(120,50){\line(1,0){10}}
\put(120,60){\line(1,0){10}}
\put(120,70){\line(1,0){10}}
\put(120,80){\line(1,0){10}}
\put(10,0){\line(0,1){10}}
\put(10,10){\line(0,1){10}}
\put(10,20){\line(0,1){10}}
\put(10,30){\line(0,1){10}}
\put(10,40){\line(0,1){10}}
\put(10,50){\line(0,1){10}}
\put(10,60){\line(0,1){10}}
\put(10,70){\line(0,1){10}}
\put(10,80){\line(0,1){10}}
\put(20,0){\line(0,1){10}}
\put(20,10){\line(0,1){10}}
\put(20,20){\line(0,1){10}}
\put(20,30){\line(0,1){10}}
\put(20,40){\line(0,1){10}}
\put(20,50){\line(0,1){10}}
\put(20,60){\line(0,1){10}}
\put(20,70){\line(0,1){10}}
\put(20,80){\line(0,1){10}}
\put(30,0){\line(0,1){10}}
\put(30,10){\line(0,1){10}}
\put(30,20){\line(0,1){10}}
\put(30,30){\line(0,1){10}}
\put(30,40){\line(0,1){10}}
\put(30,50){\line(0,1){10}}
\put(30,60){\line(0,1){10}}
\put(30,70){\line(0,1){10}}
\put(30,80){\line(0,1){10}}
\put(40,0){\line(0,1){10}}
\put(40,10){\line(0,1){10}}
\put(40,20){\line(0,1){10}}
\put(40,30){\line(0,1){10}}
\put(40,40){\line(0,1){10}}
\put(40,50){\line(0,1){10}}
\put(40,60){\line(0,1){10}}
\put(40,70){\line(0,1){10}}
\put(40,80){\line(0,1){10}}
\put(50,0){\line(0,1){10}}
\put(50,10){\line(0,1){10}}
\put(50,20){\line(0,1){10}}
\put(50,30){\line(0,1){10}}
\put(50,40){\line(0,1){10}}
\put(50,50){\line(0,1){10}}
\put(50,60){\line(0,1){10}}
\put(50,70){\line(0,1){10}}
\put(50,80){\line(0,1){10}}
\put(60,0){\line(0,1){10}}
\put(60,10){\line(0,1){10}}
\put(60,20){\line(0,1){10}}
\put(60,30){\line(0,1){10}}
\put(60,40){\line(0,1){10}}
\put(60,50){\line(0,1){10}}
\put(60,60){\line(0,1){10}}
\put(60,70){\line(0,1){10}}
\put(60,80){\line(0,1){10}}
\put(70,0){\line(0,1){10}}
\put(70,10){\line(0,1){10}}
\put(70,20){\line(0,1){10}}
\put(70,30){\line(0,1){10}}
\put(70,40){\line(0,1){10}}
\put(70,50){\line(0,1){10}}
\put(70,60){\line(0,1){10}}
\put(70,70){\line(0,1){10}}
\put(70,80){\line(0,1){10}}
\put(80,0){\line(0,1){10}}
\put(80,10){\line(0,1){10}}
\put(80,20){\line(0,1){10}}
\put(80,30){\line(0,1){10}}
\put(80,40){\line(0,1){10}}
\put(80,50){\line(0,1){10}}
\put(80,60){\line(0,1){10}}
\put(80,70){\line(0,1){10}}
\put(80,80){\line(0,1){10}}
\put(90,0){\line(0,1){10}}
\put(90,10){\line(0,1){10}}
\put(90,20){\line(0,1){10}}
\put(90,30){\line(0,1){10}}
\put(90,40){\line(0,1){10}}
\put(90,50){\line(0,1){10}}
\put(90,60){\line(0,1){10}}
\put(90,70){\line(0,1){10}}
\put(90,80){\line(0,1){10}}
\put(100,0){\line(0,1){10}}
\put(100,10){\line(0,1){10}}
\put(100,20){\line(0,1){10}}
\put(100,30){\line(0,1){10}}
\put(100,40){\line(0,1){10}}
\put(100,50){\line(0,1){10}}
\put(100,60){\line(0,1){10}}
\put(100,70){\line(0,1){10}}
\put(100,80){\line(0,1){10}}
\put(110,0){\line(0,1){10}}
\put(110,10){\line(0,1){10}}
\put(110,20){\line(0,1){10}}
\put(110,30){\line(0,1){10}}
\put(110,40){\line(0,1){10}}
\put(110,50){\line(0,1){10}}
\put(110,60){\line(0,1){10}}
\put(110,70){\line(0,1){10}}
\put(110,80){\line(0,1){10}}
\put(120,0){\line(0,1){10}}
\put(120,10){\line(0,1){10}}
\put(120,20){\line(0,1){10}}
\put(120,30){\line(0,1){10}}
\put(120,40){\line(0,1){10}}
\put(120,50){\line(0,1){10}}
\put(120,60){\line(0,1){10}}
\put(120,70){\line(0,1){10}}
\put(120,80){\line(0,1){10}}
\put(18,8){\framebox(4,4)[cc]{}}
\put(38,18){\framebox(4,4)[cc]{}}
\put(28,38){\framebox(4,4)[cc]{}}
\put(8,28){\framebox(4,4)[cc]{}}
\put(48,48){\framebox(4,4)[cc]{}}
\put(68,58){\framebox(4,4)[cc]{}}
\put(98,48){\framebox(4,4)[cc]{}}
\put(118,58){\framebox(4,4)[cc]{}}
\put(18,58){\framebox(4,4)[cc]{}}
\put(8,78){\framebox(4,4)[cc]{}}
\put(58,28){\framebox(4,4)[cc]{}}
\put(68,8){\framebox(4,4)[cc]{}}
\put(88,18){\framebox(4,4)[cc]{}}
\put(118,8){\framebox(4,4)[cc]{}}
\put(38,68){\framebox(4,4)[cc]{}}
\put(58,78){\framebox(4,4)[cc]{}}
\put(88,68){\framebox(4,4)[cc]{}}
\put(108,78){\framebox(4,4)[cc]{}}
\put(58,28){\framebox(4,4)[cc]{}}
\put(68,8){\framebox(4,4)[cc]{}}
\put(88,18){\framebox(4,4)[cc]{}}
\put(118,8){\framebox(4,4)[cc]{}}
\put(78,38){\framebox(4,4)[cc]{}}
\put(108,28){\framebox(4,4)[cc]{}}
\put(30,10){\circle{4}}
\put(20,30){\circle{4}}
\put(50,20){\circle{4}}
\put(40,40){\circle{4}}
\put(60,50){\circle{4}}
\put(80,60){\circle{4}}
\put(110,50){\circle{4}}
\put(30,60){\circle{4}}
\put(10,50){\circle{4}}
\put(20,80){\circle{4}}
\put(70,30){\circle{4}}
\put(80,10){\circle{4}}
\put(100,20){\circle{4}}
\put(50,70){\circle{4}}
\put(70,80){\circle{4}}
\put(100,70){\circle{4}}
\put(120,80){\circle{4}}
\put(70,30){\circle{4}}
\put(80,10){\circle{4}}
\put(100,20){\circle{4}}
\put(90,40){\circle{4}}
\put(120,30){\circle{4}}
\end{picture}
\end{center}
\caption{A section of the infinite grid graph $\mathbb{Z}^2$ 
and two dominating sets 
$D_\Box$ and $D_{\kr}$.}\label{fig1}
\end{figure}

The reason why we consider graphs embedded on the torus
rather than the more conventional grid graphs $P_k^2$
is to avoid boundary effects. 
It is easy to see that $\gamma(P_k^2)=k^2/5+\Omega(k)$ \cite{gopirath},
that is, the domination number $\gamma(P_k^2)$
of $P_k^2$ deviates by a term of the order $\Omega(k)=\Omega\left(\sqrt{n(P_k^2)}\right)$
from the lower bound $n(P_k^2)/(\Delta(P_k^2)+1)$,
and this disturbing deviation 
is of the same order of magnitude 
as our lower bound on 
$\partial\gamma_G(D,D').$

The elements of $D_\Box$ and $D_{\kr}$
come in adjacent pairs $p=\{ u_\Box,u_{\kr}\}$
with $u_\Box\in D_\Box$ and $u_{\kr}\in D_{\kr}$.
We consider an infinite auxiliary graph $H_\infty$
whose vertices are these pairs,
indicated by diamonds $\diamond$ in Figure \ref{fig2},
and in which two distinct pairs 
$p=\{ u_\Box,u_{\kr}\}$
and
$p'=\{ u'_\Box,u'_{\kr}\}$
are adjacent if
$$\Big(N_{\mathbb{Z}^2}\big[u_\Box\big]\cup N_{\mathbb{Z}^2}\big[u_{\kr}\big]\Big)
\cap 
\Big(N_{\mathbb{Z}^2}\big[u'_\Box\big]\cup N_{\mathbb{Z}^2}\big[u'_{\kr}\big
]\Big)
\not=\emptyset.$$
See Figure \ref{fig2} for an illustration of $H_\infty$.

\begin{figure}[h]
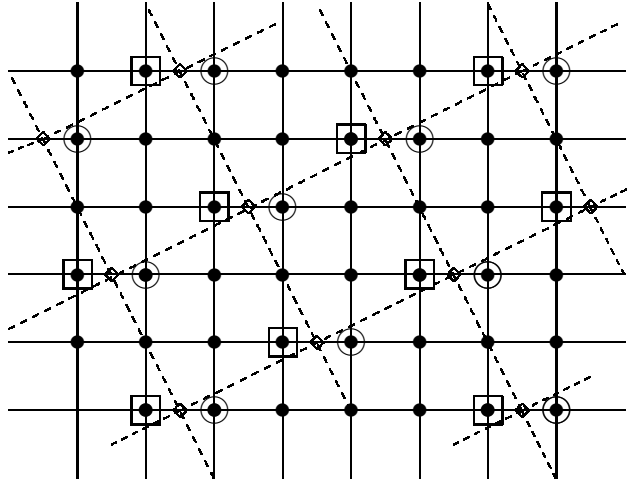

\begin{center}
\unitlength 0.9mm 
\linethickness{0.4pt}
\ifx\plotpoint\undefined\newsavebox{\plotpoint}\fi 

\end{center}
\caption{A section of the infinite auxiliary graph $H_\infty$ 
whose vertices 
are indicated by diamonds $\diamond$ and whose edges 
are indicated by dashed lines.}\label{fig2}
\end{figure}
$H_\infty$ is a tilted grid graph sitting within $\mathbb{Z}^2$
at an angle of $\arcsin \left(\frac{1}{\sqrt{5}}\right)\approx 26,56^\circ$.
The finite subgraph of $H_\infty$ corresponding to the pairs
$p=\{ u_\Box,u_{\kr}\}$
where $u_\Box$ and $u_{\kr}$ both belong to $G$
will be denoted by $H$.
At some point we want to apply to $H$
an isoperimetric inequality for the discrete torus $\mathbb{Z}^2_k$
with even $k$ 
that was shown by Bollob\'{a}s and Leader \cite{bole}.
Therefore, our construction of $G$, and hence of $H$,
ensures that $H$ is isomorphic to $\mathbb{Z}^2_k$
for some sufficiently large integer $k$ that is a multiple of $4$.
More precisely,
in order to construct $G$ and $H$,
we select, for some sufficiently large integer $k\equiv 0\mod 4$,
a $(k+1)\times (k+1)$ grid subgraph of $H_\infty$,
as illustrated in Figure \ref{fig3} for $k=8$,
and identify 
\begin{itemize}
\item the left border with the right border both from top to bottom, and
\item the top border with the bottom border both from left to right.
\end{itemize}
This yields a graph $H$ isomorphic to $\mathbb{Z}_k^2$
as well as a graph $G$ of order $5\times k^2$,
both embeddable on the torus,
for which 
$D=V(G)\cap D_\Box$
and 
$D'=V(G)\cap D_{\kr}$
are two minimum dominating sets of order $k^2$.

\begin{figure}[h]
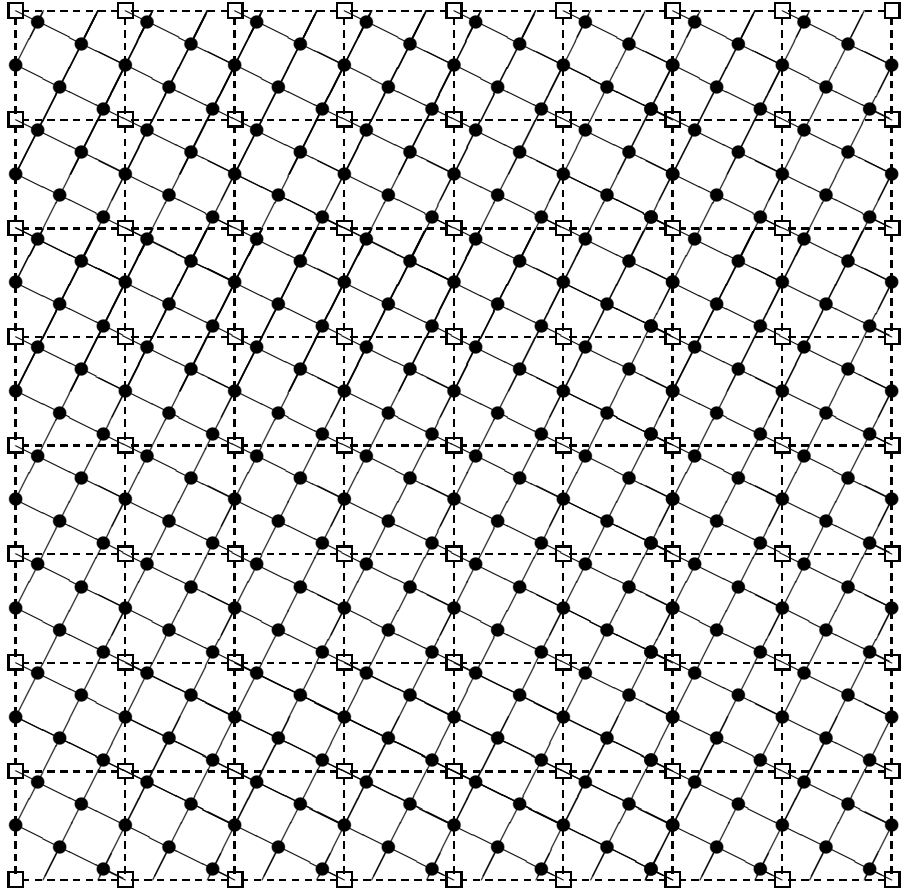

\begin{center}
\unitlength 0.9mm 
\linethickness{0.4pt}
\ifx\plotpoint\undefined\newsavebox{\plotpoint}\fi 

\end{center}
\caption{A $9\times 9$ grid subgraph of $H_\infty$.
Identifying 
the left border with the right border both from top to bottom, and
the top border with the bottom border both from left to right
yields $H$ isomorphic to $\mathbb{Z}_8^2$
and a graph $G$ of order $5\times 8^2$
embeddable on the torus
for which 
$D=V(G)\cap D_\Box$
and 
$D'=V(G)\cap D_{\kr}$
are two minimum dominating sets in $G$ both of order $8^2$.}\label{fig3}
\end{figure}
Now, let 
$$D_0,\ldots,D_\ell$$ 
be a sequence of dominating sets in $G$ such that 
$D=D_0$, 
$D'=D_\ell$, 
and $D_{i-1}$ is adjacent to $D_i$ for every $i\in [\ell]$, 
where $[\ell]=\{ 1,\ldots,\ell\}$.
For every pair $p=\{ u_\Box,u_{\kr}\}$ that is a vertex of $H$,
we say that $p$ is 
\begin{itemize}
\item of {\it type ${\rm left}$ in $D_i$} if $\{ u_\Box,u_{\kr}\}\cap D_i=\{ u_\Box\}$,
\item of {\it type ${\rm right}$ in $D_i$} if $\{ u_\Box,u_{\kr}\}\cap D_i=\{ u_{\kr}\}$,
\item of {\it type $0$ in $D_i$} if $\{ u_\Box,u_{\kr}\}\cap D_i=\emptyset$, and 
\item of {\it type $2$ in $D_i$} if $\{ u_\Box,u_{\kr}\}\cap D_i=\{ u_\Box,u_{\kr}\}$, respectively.
\end{itemize}
See Figure \ref{fig4} for an illustration of these four possibilities.

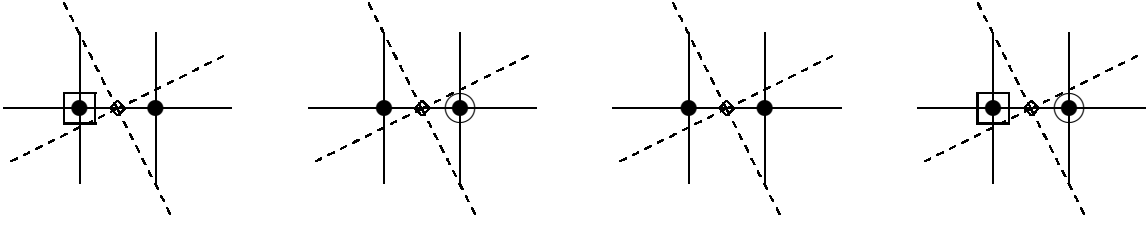
\begin{figure}[h]
\begin{center}
\unitlength 1mm 
\linethickness{0.4pt}
\ifx\plotpoint\undefined\newsavebox{\plotpoint}\fi 
\hfill$\mbox{}$
\unitlength 1mm 
\linethickness{0.4pt}
\ifx\plotpoint\undefined\newsavebox{\plotpoint}\fi 
\begin{picture}(29,30)(0,0)
\put(9,15){\circle*{2}}
\put(19,15){\circle*{2}}
\put(-1,15){\line(1,0){10}}
\put(9,15){\line(1,0){10}}
\put(19,15){\line(1,0){10}}
\put(9,5){\line(0,1){10}}
\put(9,15){\line(0,1){10}}
\put(19,5){\line(0,1){10}}
\put(19,15){\line(0,1){10}}
\put(7,13){\framebox(4,4)[cc]{}}
\multiput(13,15)(.0333333,.0333333){30}{\line(0,1){.0333333}}
\multiput(14,16)(.0333333,-.0333333){30}{\line(0,-1){.0333333}}
\multiput(15,15)(-.0333333,-.0333333){30}{\line(0,-1){.0333333}}
\multiput(14,14)(-.0333333,.0333333){30}{\line(0,1){.0333333}}
\multiput(13.93,14.93)(.0633484,.0316742){13}{\line(1,0){.0633484}}
\multiput(15.577,15.753)(.0633484,.0316742){13}{\line(1,0){.0633484}}
\multiput(17.224,16.577)(.0633484,.0316742){13}{\line(1,0){.0633484}}
\multiput(18.871,17.4)(.0633484,.0316742){13}{\line(1,0){.0633484}}
\multiput(20.518,18.224)(.0633484,.0316742){13}{\line(1,0){.0633484}}
\multiput(22.165,19.047)(.0633484,.0316742){13}{\line(1,0){.0633484}}
\multiput(23.812,19.871)(.0633484,.0316742){13}{\line(1,0){.0633484}}
\multiput(25.459,20.694)(.0633484,.0316742){13}{\line(1,0){.0633484}}
\multiput(27.106,21.518)(.0633484,.0316742){13}{\line(1,0){.0633484}}
\multiput(-.07,7.93)(.0633484,.0316742){13}{\line(1,0){.0633484}}
\multiput(1.577,8.753)(.0633484,.0316742){13}{\line(1,0){.0633484}}
\multiput(3.224,9.577)(.0633484,.0316742){13}{\line(1,0){.0633484}}
\multiput(4.871,10.4)(.0633484,.0316742){13}{\line(1,0){.0633484}}
\multiput(6.518,11.224)(.0633484,.0316742){13}{\line(1,0){.0633484}}
\multiput(8.165,12.047)(.0633484,.0316742){13}{\line(1,0){.0633484}}
\multiput(9.812,12.871)(.0633484,.0316742){13}{\line(1,0){.0633484}}
\multiput(11.459,13.694)(.0633484,.0316742){13}{\line(1,0){.0633484}}
\multiput(13.106,14.518)(.0633484,.0316742){13}{\line(1,0){.0633484}}
\multiput(13.93,14.93)(-.0316742,.0633484){13}{\line(0,1){.0633484}}
\multiput(13.106,16.577)(-.0316742,.0633484){13}{\line(0,1){.0633484}}
\multiput(12.283,18.224)(-.0316742,.0633484){13}{\line(0,1){.0633484}}
\multiput(11.459,19.871)(-.0316742,.0633484){13}{\line(0,1){.0633484}}
\multiput(10.636,21.518)(-.0316742,.0633484){13}{\line(0,1){.0633484}}
\multiput(9.812,23.165)(-.0316742,.0633484){13}{\line(0,1){.0633484}}
\multiput(8.989,24.812)(-.0316742,.0633484){13}{\line(0,1){.0633484}}
\multiput(8.165,26.459)(-.0316742,.0633484){13}{\line(0,1){.0633484}}
\multiput(7.341,28.106)(-.0316742,.0633484){13}{\line(0,1){.0633484}}
\multiput(20.93,.93)(-.0316742,.0633484){13}{\line(0,1){.0633484}}
\multiput(20.106,2.577)(-.0316742,.0633484){13}{\line(0,1){.0633484}}
\multiput(19.283,4.224)(-.0316742,.0633484){13}{\line(0,1){.0633484}}
\multiput(18.459,5.871)(-.0316742,.0633484){13}{\line(0,1){.0633484}}
\multiput(17.636,7.518)(-.0316742,.0633484){13}{\line(0,1){.0633484}}
\multiput(16.812,9.165)(-.0316742,.0633484){13}{\line(0,1){.0633484}}
\multiput(15.989,10.812)(-.0316742,.0633484){13}{\line(0,1){.0633484}}
\multiput(15.165,12.459)(-.0316742,.0633484){13}{\line(0,1){.0633484}}
\multiput(14.341,14.106)(-.0316742,.0633484){13}{\line(0,1){.0633484}}
\end{picture}
\hfill
\unitlength 1mm 
\linethickness{0.4pt}
\ifx\plotpoint\undefined\newsavebox{\plotpoint}\fi 
\begin{picture}(29,30)(0,0)
\put(9,15){\circle*{2}}
\put(19,15){\circle*{2}}
\put(-1,15){\line(1,0){10}}
\put(9,15){\line(1,0){10}}
\put(19,15){\line(1,0){10}}
\put(9,5){\line(0,1){10}}
\put(9,15){\line(0,1){10}}
\put(19,5){\line(0,1){10}}
\put(19,15){\line(0,1){10}}
\put(19,15){\circle{4}}
\multiput(13,15)(.0333333,.0333333){30}{\line(0,1){.0333333}}
\multiput(14,16)(.0333333,-.0333333){30}{\line(0,-1){.0333333}}
\multiput(15,15)(-.0333333,-.0333333){30}{\line(0,-1){.0333333}}
\multiput(14,14)(-.0333333,.0333333){30}{\line(0,1){.0333333}}
\multiput(13.93,14.93)(.0633484,.0316742){13}{\line(1,0){.0633484}}
\multiput(15.577,15.753)(.0633484,.0316742){13}{\line(1,0){.0633484}}
\multiput(17.224,16.577)(.0633484,.0316742){13}{\line(1,0){.0633484}}
\multiput(18.871,17.4)(.0633484,.0316742){13}{\line(1,0){.0633484}}
\multiput(20.518,18.224)(.0633484,.0316742){13}{\line(1,0){.0633484}}
\multiput(22.165,19.047)(.0633484,.0316742){13}{\line(1,0){.0633484}}
\multiput(23.812,19.871)(.0633484,.0316742){13}{\line(1,0){.0633484}}
\multiput(25.459,20.694)(.0633484,.0316742){13}{\line(1,0){.0633484}}
\multiput(27.106,21.518)(.0633484,.0316742){13}{\line(1,0){.0633484}}
\multiput(-.07,7.93)(.0633484,.0316742){13}{\line(1,0){.0633484}}
\multiput(1.577,8.753)(.0633484,.0316742){13}{\line(1,0){.0633484}}
\multiput(3.224,9.577)(.0633484,.0316742){13}{\line(1,0){.0633484}}
\multiput(4.871,10.4)(.0633484,.0316742){13}{\line(1,0){.0633484}}
\multiput(6.518,11.224)(.0633484,.0316742){13}{\line(1,0){.0633484}}
\multiput(8.165,12.047)(.0633484,.0316742){13}{\line(1,0){.0633484}}
\multiput(9.812,12.871)(.0633484,.0316742){13}{\line(1,0){.0633484}}
\multiput(11.459,13.694)(.0633484,.0316742){13}{\line(1,0){.0633484}}
\multiput(13.106,14.518)(.0633484,.0316742){13}{\line(1,0){.0633484}}
\multiput(13.93,14.93)(-.0316742,.0633484){13}{\line(0,1){.0633484}}
\multiput(13.106,16.577)(-.0316742,.0633484){13}{\line(0,1){.0633484}}
\multiput(12.283,18.224)(-.0316742,.0633484){13}{\line(0,1){.0633484}}
\multiput(11.459,19.871)(-.0316742,.0633484){13}{\line(0,1){.0633484}}
\multiput(10.636,21.518)(-.0316742,.0633484){13}{\line(0,1){.0633484}}
\multiput(9.812,23.165)(-.0316742,.0633484){13}{\line(0,1){.0633484}}
\multiput(8.989,24.812)(-.0316742,.0633484){13}{\line(0,1){.0633484}}
\multiput(8.165,26.459)(-.0316742,.0633484){13}{\line(0,1){.0633484}}
\multiput(7.341,28.106)(-.0316742,.0633484){13}{\line(0,1){.0633484}}
\multiput(20.93,.93)(-.0316742,.0633484){13}{\line(0,1){.0633484}}
\multiput(20.106,2.577)(-.0316742,.0633484){13}{\line(0,1){.0633484}}
\multiput(19.283,4.224)(-.0316742,.0633484){13}{\line(0,1){.0633484}}
\multiput(18.459,5.871)(-.0316742,.0633484){13}{\line(0,1){.0633484}}
\multiput(17.636,7.518)(-.0316742,.0633484){13}{\line(0,1){.0633484}}
\multiput(16.812,9.165)(-.0316742,.0633484){13}{\line(0,1){.0633484}}
\multiput(15.989,10.812)(-.0316742,.0633484){13}{\line(0,1){.0633484}}
\multiput(15.165,12.459)(-.0316742,.0633484){13}{\line(0,1){.0633484}}
\multiput(14.341,14.106)(-.0316742,.0633484){13}{\line(0,1){.0633484}}
\end{picture}
\hfill
\unitlength 1mm 
\linethickness{0.4pt}
\ifx\plotpoint\undefined\newsavebox{\plotpoint}\fi 
\begin{picture}(29,30)(0,0)
\put(9,15){\circle*{2}}
\put(19,15){\circle*{2}}
\put(-1,15){\line(1,0){10}}
\put(9,15){\line(1,0){10}}
\put(19,15){\line(1,0){10}}
\put(9,5){\line(0,1){10}}
\put(9,15){\line(0,1){10}}
\put(19,5){\line(0,1){10}}
\put(19,15){\line(0,1){10}}
\multiput(13,15)(.0333333,.0333333){30}{\line(0,1){.0333333}}
\multiput(14,16)(.0333333,-.0333333){30}{\line(0,-1){.0333333}}
\multiput(15,15)(-.0333333,-.0333333){30}{\line(0,-1){.0333333}}
\multiput(14,14)(-.0333333,.0333333){30}{\line(0,1){.0333333}}
\multiput(13.93,14.93)(.0633484,.0316742){13}{\line(1,0){.0633484}}
\multiput(15.577,15.753)(.0633484,.0316742){13}{\line(1,0){.0633484}}
\multiput(17.224,16.577)(.0633484,.0316742){13}{\line(1,0){.0633484}}
\multiput(18.871,17.4)(.0633484,.0316742){13}{\line(1,0){.0633484}}
\multiput(20.518,18.224)(.0633484,.0316742){13}{\line(1,0){.0633484}}
\multiput(22.165,19.047)(.0633484,.0316742){13}{\line(1,0){.0633484}}
\multiput(23.812,19.871)(.0633484,.0316742){13}{\line(1,0){.0633484}}
\multiput(25.459,20.694)(.0633484,.0316742){13}{\line(1,0){.0633484}}
\multiput(27.106,21.518)(.0633484,.0316742){13}{\line(1,0){.0633484}}
\multiput(-.07,7.93)(.0633484,.0316742){13}{\line(1,0){.0633484}}
\multiput(1.577,8.753)(.0633484,.0316742){13}{\line(1,0){.0633484}}
\multiput(3.224,9.577)(.0633484,.0316742){13}{\line(1,0){.0633484}}
\multiput(4.871,10.4)(.0633484,.0316742){13}{\line(1,0){.0633484}}
\multiput(6.518,11.224)(.0633484,.0316742){13}{\line(1,0){.0633484}}
\multiput(8.165,12.047)(.0633484,.0316742){13}{\line(1,0){.0633484}}
\multiput(9.812,12.871)(.0633484,.0316742){13}{\line(1,0){.0633484}}
\multiput(11.459,13.694)(.0633484,.0316742){13}{\line(1,0){.0633484}}
\multiput(13.106,14.518)(.0633484,.0316742){13}{\line(1,0){.0633484}}
\multiput(13.93,14.93)(-.0316742,.0633484){13}{\line(0,1){.0633484}}
\multiput(13.106,16.577)(-.0316742,.0633484){13}{\line(0,1){.0633484}}
\multiput(12.283,18.224)(-.0316742,.0633484){13}{\line(0,1){.0633484}}
\multiput(11.459,19.871)(-.0316742,.0633484){13}{\line(0,1){.0633484}}
\multiput(10.636,21.518)(-.0316742,.0633484){13}{\line(0,1){.0633484}}
\multiput(9.812,23.165)(-.0316742,.0633484){13}{\line(0,1){.0633484}}
\multiput(8.989,24.812)(-.0316742,.0633484){13}{\line(0,1){.0633484}}
\multiput(8.165,26.459)(-.0316742,.0633484){13}{\line(0,1){.0633484}}
\multiput(7.341,28.106)(-.0316742,.0633484){13}{\line(0,1){.0633484}}
\multiput(20.93,.93)(-.0316742,.0633484){13}{\line(0,1){.0633484}}
\multiput(20.106,2.577)(-.0316742,.0633484){13}{\line(0,1){.0633484}}
\multiput(19.283,4.224)(-.0316742,.0633484){13}{\line(0,1){.0633484}}
\multiput(18.459,5.871)(-.0316742,.0633484){13}{\line(0,1){.0633484}}
\multiput(17.636,7.518)(-.0316742,.0633484){13}{\line(0,1){.0633484}}
\multiput(16.812,9.165)(-.0316742,.0633484){13}{\line(0,1){.0633484}}
\multiput(15.989,10.812)(-.0316742,.0633484){13}{\line(0,1){.0633484}}
\multiput(15.165,12.459)(-.0316742,.0633484){13}{\line(0,1){.0633484}}
\multiput(14.341,14.106)(-.0316742,.0633484){13}{\line(0,1){.0633484}}
\end{picture}
\hfill
\unitlength 1mm 
\linethickness{0.4pt}
\ifx\plotpoint\undefined\newsavebox{\plotpoint}\fi 
\begin{picture}(29,30)(0,0)
\put(9,15){\circle*{2}}
\put(19,15){\circle*{2}}
\put(-1,15){\line(1,0){10}}
\put(9,15){\line(1,0){10}}
\put(19,15){\line(1,0){10}}
\put(9,5){\line(0,1){10}}
\put(9,15){\line(0,1){10}}
\put(19,5){\line(0,1){10}}
\put(19,15){\line(0,1){10}}
\put(7,13){\framebox(4,4)[cc]{}}
\put(19,15){\circle{4}}
\multiput(13,15)(.0333333,.0333333){30}{\line(0,1){.0333333}}
\multiput(14,16)(.0333333,-.0333333){30}{\line(0,-1){.0333333}}
\multiput(15,15)(-.0333333,-.0333333){30}{\line(0,-1){.0333333}}
\multiput(14,14)(-.0333333,.0333333){30}{\line(0,1){.0333333}}
\multiput(13.93,14.93)(.0633484,.0316742){13}{\line(1,0){.0633484}}
\multiput(15.577,15.753)(.0633484,.0316742){13}{\line(1,0){.0633484}}
\multiput(17.224,16.577)(.0633484,.0316742){13}{\line(1,0){.0633484}}
\multiput(18.871,17.4)(.0633484,.0316742){13}{\line(1,0){.0633484}}
\multiput(20.518,18.224)(.0633484,.0316742){13}{\line(1,0){.0633484}}
\multiput(22.165,19.047)(.0633484,.0316742){13}{\line(1,0){.0633484}}
\multiput(23.812,19.871)(.0633484,.0316742){13}{\line(1,0){.0633484}}
\multiput(25.459,20.694)(.0633484,.0316742){13}{\line(1,0){.0633484}}
\multiput(27.106,21.518)(.0633484,.0316742){13}{\line(1,0){.0633484}}
\multiput(-.07,7.93)(.0633484,.0316742){13}{\line(1,0){.0633484}}
\multiput(1.577,8.753)(.0633484,.0316742){13}{\line(1,0){.0633484}}
\multiput(3.224,9.577)(.0633484,.0316742){13}{\line(1,0){.0633484}}
\multiput(4.871,10.4)(.0633484,.0316742){13}{\line(1,0){.0633484}}
\multiput(6.518,11.224)(.0633484,.0316742){13}{\line(1,0){.0633484}}
\multiput(8.165,12.047)(.0633484,.0316742){13}{\line(1,0){.0633484}}
\multiput(9.812,12.871)(.0633484,.0316742){13}{\line(1,0){.0633484}}
\multiput(11.459,13.694)(.0633484,.0316742){13}{\line(1,0){.0633484}}
\multiput(13.106,14.518)(.0633484,.0316742){13}{\line(1,0){.0633484}}
\multiput(13.93,14.93)(-.0316742,.0633484){13}{\line(0,1){.0633484}}
\multiput(13.106,16.577)(-.0316742,.0633484){13}{\line(0,1){.0633484}}
\multiput(12.283,18.224)(-.0316742,.0633484){13}{\line(0,1){.0633484}}
\multiput(11.459,19.871)(-.0316742,.0633484){13}{\line(0,1){.0633484}}
\multiput(10.636,21.518)(-.0316742,.0633484){13}{\line(0,1){.0633484}}
\multiput(9.812,23.165)(-.0316742,.0633484){13}{\line(0,1){.0633484}}
\multiput(8.989,24.812)(-.0316742,.0633484){13}{\line(0,1){.0633484}}
\multiput(8.165,26.459)(-.0316742,.0633484){13}{\line(0,1){.0633484}}
\multiput(7.341,28.106)(-.0316742,.0633484){13}{\line(0,1){.0633484}}
\multiput(20.93,.93)(-.0316742,.0633484){13}{\line(0,1){.0633484}}
\multiput(20.106,2.577)(-.0316742,.0633484){13}{\line(0,1){.0633484}}
\multiput(19.283,4.224)(-.0316742,.0633484){13}{\line(0,1){.0633484}}
\multiput(18.459,5.871)(-.0316742,.0633484){13}{\line(0,1){.0633484}}
\multiput(17.636,7.518)(-.0316742,.0633484){13}{\line(0,1){.0633484}}
\multiput(16.812,9.165)(-.0316742,.0633484){13}{\line(0,1){.0633484}}
\multiput(15.989,10.812)(-.0316742,.0633484){13}{\line(0,1){.0633484}}
\multiput(15.165,12.459)(-.0316742,.0633484){13}{\line(0,1){.0633484}}
\multiput(14.341,14.106)(-.0316742,.0633484){13}{\line(0,1){.0633484}}
\end{picture}
\hfill$\mbox{}$
\end{center}
\caption{
From left to right,
a pair $p=\{ u_\Box,u_{\kr}\}$ of 
type ${\rm left}$,
${\rm right}$,
$0$, 
and $2$.}\label{fig4}
\end{figure}
Let $[\ell]_0=\{ 0,1,\ldots,\ell\}$.
For an integer $i\in [\ell]_0$ and a type $t\in \{ {\rm left},{\rm right},0,2\}$, 
let $P(i,t)$ be the set of all vertices 
$p=\{ u_\Box,u_{\kr}\}$ of $H$
that are of type $t$ in $D_i$,
and let $n(i,t)=|P(i,t)|$.
Trivially,
$$
n(0,t) =
\begin{cases}
k^2, & t={\rm left},\\
0, & \mbox{otherwise}.
\end{cases}
\,\,\,\,\,\,\,\,\,\,\,\,\,\,\mbox{ and }\,\,\,\,\,\,\,\,\,\,\,\,\,\,
n(\ell,t) =
\begin{cases}
k^2, & t={\rm right},\\
0, & \mbox{otherwise}.
\end{cases}
$$
Furthermore, for every $i\in [\ell]$ and type $t$, 
since $D_i$ arises from $D_{i-1}$ by removing or adding a single vertex,
we have
$$
\Big|n(i,t)-n(i-1,t)\Big|\leq 1.
$$
Since $k$ is a multiple of $4$,
$k^2/8$ is an integer.
Hence, if $j\in [\ell]_0$ is the smallest index such that 
$$n(j,{\rm left})\leq \frac{7k^2}{8},$$
then 
$$n(j,{\rm left})=\frac{7k^2}{8}
\,\,\,\,\,\,\,\,\,\,\,\,\,\,\mbox{ and }\,\,\,\,\,\,\,\,\,\,\,\,\,\,
n(j,{\rm right})+n(j,0)+n(j,2)=k^2-\frac{7k^2}{8}=\frac{k^2}{8}.$$
In order to complete the proof, 
we will show that the dominating set
$$D^*=D_j$$
has cardinality at least $k^2+\Omega(k)$.
This is done by showing that $D^*$ contains $\Omega(k)$
so-called inefficient vertices,
where a vertex $u$ in $D^*$ is {\it inefficient} 
if there is some vertex $v$ in $D^*\setminus \{ u\}$ with 
$N_G[u]\cap N_G[v]\not=\emptyset$.

Let $P(t)=P(j,t)$ 
and $n(t)=n(j,t)$ for every type $t$.
Let $P'({\rm left})$ be the subset of $P({\rm left})$
containing all vertices $p=\{ u_\Box,u_{\kr}\}$ of $H$
whose four neighbors in $H$ are all not of type ${\rm left}$.
Double counting the edges of $H$ between 
$P'({\rm left})$
and 
$P({\rm right})\cup P(0)\cup P(2)$
implies
$$\big|P'({\rm left})\big|\leq \big|P({\rm right})\cup P(0)\cup P(2)\big|=
\frac{k^2}{8},$$
and, hence, 
the set $P^*=P({\rm left})\setminus P'({\rm left})$ satisfies 
\begin{eqnarray}\label{e1}
\frac{6k^2}{8}\leq |P^*|\leq \frac{7k^2}{8}.
\end{eqnarray}
Let $P^{**}$ be the set of vertices in $P^*$
that have a neighbor outside of $P^*$.
Note that, by construction, 
every vertex in $P^{**}$ 
belongs to $P({\rm left})$, 
has a neighbor in $P({\rm left})$
as well as in $P({\rm right})\cup P(0)\cup P(2)$, 
but has no neighbor in $P'({\rm left})$.

Our next goal is the following.

\begin{lemma}\label{lemma1}
$|P^{**}|\geq \frac{k}{20}+O(1)$.
\end{lemma}
\begin{proof}
For a set $S$ of vertices of $H$, 
let $N_H(S)=\{ u\in V(H)\setminus S:N_H(u)\cap S\not=\emptyset\}$.
Let $u$ be any vertex of $H$,
and, for a non-negative integer $r$,
let $B(r)=\{ v\in V(H):{\rm dist}_H(u,v)\leq r\}$.
Since $H$ is vertex-transitive,
it follows that $|B(r)|$ is independent of the choice of $u$.
Bollob\'{a}s and Leader \cite{bole} showed that, 
if $|S|=|B(r)|$ for some non-negative integer $r$, then 
$$|N_H(S)|\geq |N_H(B(r))|=|B(r+1)|-|B(r)|.$$
Note that their result only applies to sets whose cardinality
is in $\{ |B(0)|,|B(1)|,|B(2)|,\ldots\}$,
which causes some technicalities in our proof.
Let the non-negative integer $r^*$ be such that 
$$|B(r^*-1)|\leq |P^*|\leq |B(r^*)|.$$
Combining some simple geometric considerations
illustrated in Figure \ref{fig5}
with (\ref{e1}) implies the existence of some 
$c_1\in \left[1/4,\sqrt{2}/4\right]$ such that
\begin{eqnarray} 
r^* & = & (1-c_1)k+O(1),\nonumber\\
|N_H(B(r^*))| &=& |B(r^*+1)|-|B(r^*)| = 4c_1k+O(1),\mbox{ and}\label{e2}\\
|N_H(B(r^*-1))| &=& |B(r^*)|-|B(r^*-1)| = 4c_1k+O(1).\nonumber
\end{eqnarray} 

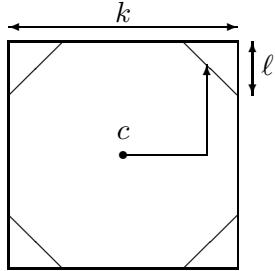
\begin{figure}[h]
\begin{center}
\unitlength 1mm 
\linethickness{0.4pt}
\ifx\plotpoint\undefined\newsavebox{\plotpoint}\fi 
\begin{picture}(34,33)(0,0)
\put(0,0){\framebox(30,30)[cc]{}}
\put(0,32){\vector(-1,0){.07}}\put(30,32){\vector(1,0){.07}}\put(30,32){\line(-1,0){30}}
\put(32,23){\vector(0,-1){.07}}\put(32,30){\vector(0,1){.07}}\put(32,30){\line(0,-1){7}}
\put(30,23){\line(-1,1){7}}
\put(7,30){\line(-1,-1){7}}
\put(30,7){\line(-1,-1){7}}
\put(7,0){\line(-1,1){7}}
\put(15,15){\circle*{1}}
\put(15,15){\line(1,0){11}}
\put(26,15){\vector(0,1){12}}
\put(15,34){\makebox(0,0)[cc]{$k$}}
\put(34,27){\makebox(0,0)[cc]{$\ell$}}
\put(15,18){\makebox(0,0)[cc]{$c$}}
\end{picture}
\end{center}
\caption{
Within a $k\times k$ square $Q$ in $\mathbb{R}^2$,
the set of points in $Q$ at Manhattan distance at most $k-\ell$
from the center point $c$ of $Q$ 
for some $\ell<k/2$ has area 
$k^2-2\ell^2$.
If $k^2-2\ell^2\in \left[6k^2/8,7k^2/8\right]$,
then $\ell/k\in \left[1/4,\sqrt{2}/4\right]$.}\label{fig5}
\end{figure}
If $S$ is a set of vertices of $H$, and $u$ is a vertex of $H$ outside of $S$, then 
\begin{eqnarray}\label{eh}
|N_H(S)|+4\geq |N_H(S\cup \{ u\})|\geq |N_H(S)|-1.
\end{eqnarray}
If $|B(r^*)|-|P^*|\leq \frac{4c_1k}{5}$,
then adding $|B(r^*)|-|P^*|$ vertices to $P^*$ 
yields a set $P'$ of order $|B(r^*)|$ with 
\begin{eqnarray*}
|N_H(P^*)|
&\stackrel{(\ref{eh})}{\geq} & |N_H(P')|-4\big(|B(r^*)|-|P^*|\big)\\
&\geq &|N_H(P')|-\frac{16c_1k}{5}\\
&\geq &|N_H(B(r^*))|-\frac{16c_1k}{5}\\
&\stackrel{(\ref{e2})}{\geq} & 4c_1k\left(1-\frac{4}{5}\right)+O(1)\\
&\geq &\frac{4c_1k}{5}+O(1).
\end{eqnarray*}
Conversely, if 
$|B(r^*)|-|P^*|>\frac{4c_1k}{5}$,
then (\ref{e2}) implies 
$|P^*|-|B(r^*-1)|\leq \frac{16c_1k}{5}+O(1)$,
and removing $|P^*|-|B(r^*-1)|$ vertices from $P^*$
yields a set $P''$ of order $|B(r^*-1)|$ with 
\begin{eqnarray*}
|N_H(P^*)|
&\stackrel{(\ref{eh})}{\geq} &|N_H(P'')|-\left(\frac{16c_1k}{5}+O(1)\right)\\
&\geq & |N_H(P'')|-\frac{16c_1k}{5}+O(1)\\
&\geq & |N_H(B(r^*-1))|-\frac{16c_1k}{5}+O(1)\\
&\stackrel{(\ref{e2})}{\geq} & 4c_1k\left(1-\frac{4}{5}\right)+O(1)\\
&\geq &\frac{4c_1k}{5}+O(1).
\end{eqnarray*}
Altogether, it follows that 
$$|N_H(P^*)|\geq \frac{4c_1k}{5}+O(1)\geq \frac{k}{5}+O(1).$$
Since
$|P^{**}|\geq |N_H(P^*)|/4$,
the proof of Lemma \ref{lemma1} is complete. 
\end{proof}

Now, let $p=\{ u_\Box,u_{\kr}\}$ be a vertex in $P^{**}$,
illustrated as the pair $\{ (2,2),(3,2)\}$ in Figure \ref{fig6}.

\begin{figure}[h]
\begin{center}
\unitlength 1mm 
\linethickness{0.4pt}
\ifx\plotpoint\undefined\newsavebox{\plotpoint}\fi 
\begin{picture}(71,60)(0,0)
\put(10,10){\circle*{2}}
\put(20,20){\circle*{2}}
\put(10,30){\circle*{2}}
\put(20,10){\circle*{2}}
\put(20,20){\circle*{2}}
\put(20,30){\circle*{2}}
\put(20,40){\circle*{2}}
\put(20,50){\circle*{2}}
\put(30,10){\circle*{2}}
\put(30,20){\circle*{2}}
\put(30,30){\circle*{2}}
\put(10,20){\circle*{2}}
\put(30,40){\circle*{2}}
\put(30,50){\circle*{2}}
\put(40,10){\circle*{2}}
\put(40,20){\circle*{2}}
\put(40,30){\circle*{2}}
\put(40,40){\circle*{2}}
\put(40,50){\circle*{2}}
\put(50,10){\circle*{2}}
\put(50,20){\circle*{2}}
\put(50,30){\circle*{2}}
\put(50,40){\circle*{2}}
\put(50,50){\circle*{2}}
\put(60,30){\circle*{2}}
\put(50,40){\circle*{2}}
\put(60,50){\circle*{2}}
\put(60,40){\circle*{2}}
\put(0,10){\line(1,0){10}}
\put(0,20){\line(1,0){10}}
\put(0,30){\line(1,0){10}}
\put(10,10){\line(1,0){10}}
\put(10,20){\line(1,0){10}}
\put(10,30){\line(1,0){10}}
\put(10,40){\line(1,0){10}}
\put(10,50){\line(1,0){10}}
\put(20,10){\line(1,0){10}}
\put(20,20){\line(1,0){10}}
\put(20,30){\line(1,0){10}}
\put(20,40){\line(1,0){10}}
\put(20,50){\line(1,0){10}}
\put(30,10){\line(1,0){10}}
\put(30,20){\line(1,0){10}}
\put(30,30){\line(1,0){10}}
\put(30,40){\line(1,0){10}}
\put(30,50){\line(1,0){10}}
\put(40,10){\line(1,0){10}}
\put(40,20){\line(1,0){10}}
\put(40,30){\line(1,0){10}}
\put(40,40){\line(1,0){10}}
\put(40,50){\line(1,0){10}}
\put(50,10){\line(1,0){10}}
\put(50,20){\line(1,0){10}}
\put(50,30){\line(1,0){10}}
\put(50,40){\line(1,0){10}}
\put(50,50){\line(1,0){10}}
\put(60,30){\line(1,0){10}}
\put(50,40){\line(1,0){10}}
\put(60,50){\line(1,0){10}}
\put(60,40){\line(1,0){10}}
\put(10,0){\line(0,1){10}}
\put(10,10){\line(0,1){10}}
\put(10,20){\line(0,1){10}}
\put(20,0){\line(0,1){10}}
\put(20,10){\line(0,1){10}}
\put(20,20){\line(0,1){10}}
\put(20,30){\line(0,1){10}}
\put(20,40){\line(0,1){10}}
\put(20,50){\line(0,1){10}}
\put(30,0){\line(0,1){10}}
\put(30,10){\line(0,1){10}}
\put(30,20){\line(0,1){10}}
\put(30,30){\line(0,1){10}}
\put(30,40){\line(0,1){10}}
\put(30,50){\line(0,1){10}}
\put(40,0){\line(0,1){10}}
\put(40,10){\line(0,1){10}}
\put(40,20){\line(0,1){10}}
\put(40,30){\line(0,1){10}}
\put(40,40){\line(0,1){10}}
\put(40,50){\line(0,1){10}}
\put(50,0){\line(0,1){10}}
\put(50,10){\line(0,1){10}}
\put(50,20){\line(0,1){10}}
\put(50,30){\line(0,1){10}}
\put(50,40){\line(0,1){10}}
\put(50,50){\line(0,1){10}}
\put(60,30){\line(0,1){10}}
\put(60,40){\line(0,1){10}}
\put(60,50){\line(0,1){10}}
\put(38,8){\framebox(4,4)[cc]{}}
\put(28,28){\framebox(4,4)[cc]{}}
\put(8,18){\framebox(4,4)[cc]{}}
\put(48,38){\framebox(4,4)[cc]{}}
\put(18,48){\framebox(4,4)[cc]{}}
\put(20,20){\circle{4}}
\put(50,10){\circle{4}}
\put(40,30){\circle{4}}
\put(60,40){\circle{4}}
\put(30,50){\circle{4}}
\multiput(14,20)(.0333333,.0333333){30}{\line(0,1){.0333333}}
\multiput(44,10)(.0333333,.0333333){30}{\line(0,1){.0333333}}
\multiput(34,30)(.0333333,.0333333){30}{\line(0,1){.0333333}}
\multiput(24,50)(.0333333,.0333333){30}{\line(0,1){.0333333}}
\multiput(54,40)(.0333333,.0333333){30}{\line(0,1){.0333333}}
\multiput(15,21)(.0333333,-.0333333){30}{\line(0,-1){.0333333}}
\multiput(45,11)(.0333333,-.0333333){30}{\line(0,-1){.0333333}}
\multiput(35,31)(.0333333,-.0333333){30}{\line(0,-1){.0333333}}
\multiput(25,51)(.0333333,-.0333333){30}{\line(0,-1){.0333333}}
\multiput(55,41)(.0333333,-.0333333){30}{\line(0,-1){.0333333}}
\multiput(16,20)(-.0333333,-.0333333){30}{\line(0,-1){.0333333}}
\multiput(46,10)(-.0333333,-.0333333){30}{\line(0,-1){.0333333}}
\multiput(36,30)(-.0333333,-.0333333){30}{\line(0,-1){.0333333}}
\multiput(26,50)(-.0333333,-.0333333){30}{\line(0,-1){.0333333}}
\multiput(56,40)(-.0333333,-.0333333){30}{\line(0,-1){.0333333}}
\multiput(15,19)(-.0333333,.0333333){30}{\line(0,1){.0333333}}
\multiput(45,9)(-.0333333,.0333333){30}{\line(0,1){.0333333}}
\multiput(35,29)(-.0333333,.0333333){30}{\line(0,1){.0333333}}
\multiput(25,49)(-.0333333,.0333333){30}{\line(0,1){.0333333}}
\multiput(55,39)(-.0333333,.0333333){30}{\line(0,1){.0333333}}
\multiput(49.93,-.07)(-.0315126,.0630252){14}{\line(0,1){.0630252}}
\multiput(49.047,1.694)(-.0315126,.0630252){14}{\line(0,1){.0630252}}
\multiput(48.165,3.459)(-.0315126,.0630252){14}{\line(0,1){.0630252}}
\multiput(47.283,5.224)(-.0315126,.0630252){14}{\line(0,1){.0630252}}
\multiput(46.4,6.989)(-.0315126,.0630252){14}{\line(0,1){.0630252}}
\multiput(45.518,8.753)(-.0315126,.0630252){14}{\line(0,1){.0630252}}
\multiput(44.636,10.518)(-.0315126,.0630252){14}{\line(0,1){.0630252}}
\multiput(43.753,12.283)(-.0315126,.0630252){14}{\line(0,1){.0630252}}
\multiput(42.871,14.047)(-.0315126,.0630252){14}{\line(0,1){.0630252}}
\multiput(41.989,15.812)(-.0315126,.0630252){14}{\line(0,1){.0630252}}
\multiput(41.106,17.577)(-.0315126,.0630252){14}{\line(0,1){.0630252}}
\multiput(40.224,19.341)(-.0315126,.0630252){14}{\line(0,1){.0630252}}
\multiput(39.341,21.106)(-.0315126,.0630252){14}{\line(0,1){.0630252}}
\multiput(38.459,22.871)(-.0315126,.0630252){14}{\line(0,1){.0630252}}
\multiput(37.577,24.636)(-.0315126,.0630252){14}{\line(0,1){.0630252}}
\multiput(36.694,26.4)(-.0315126,.0630252){14}{\line(0,1){.0630252}}
\multiput(35.812,28.165)(-.0315126,.0630252){14}{\line(0,1){.0630252}}
\multiput(34.93,29.93)(-.0315126,.0630252){14}{\line(0,1){.0630252}}
\multiput(34.047,31.694)(-.0315126,.0630252){14}{\line(0,1){.0630252}}
\multiput(33.165,33.459)(-.0315126,.0630252){14}{\line(0,1){.0630252}}
\multiput(32.283,35.224)(-.0315126,.0630252){14}{\line(0,1){.0630252}}
\multiput(31.4,36.989)(-.0315126,.0630252){14}{\line(0,1){.0630252}}
\multiput(30.518,38.753)(-.0315126,.0630252){14}{\line(0,1){.0630252}}
\multiput(29.636,40.518)(-.0315126,.0630252){14}{\line(0,1){.0630252}}
\multiput(28.753,42.283)(-.0315126,.0630252){14}{\line(0,1){.0630252}}
\multiput(27.871,44.047)(-.0315126,.0630252){14}{\line(0,1){.0630252}}
\multiput(26.989,45.812)(-.0315126,.0630252){14}{\line(0,1){.0630252}}
\multiput(26.106,47.577)(-.0315126,.0630252){14}{\line(0,1){.0630252}}
\multiput(25.224,49.341)(-.0315126,.0630252){14}{\line(0,1){.0630252}}
\multiput(24.341,51.106)(-.0315126,.0630252){14}{\line(0,1){.0630252}}
\multiput(23.459,52.871)(-.0315126,.0630252){14}{\line(0,1){.0630252}}
\multiput(22.577,54.636)(-.0315126,.0630252){14}{\line(0,1){.0630252}}
\multiput(21.694,56.4)(-.0315126,.0630252){14}{\line(0,1){.0630252}}
\multiput(20.812,58.165)(-.0315126,.0630252){14}{\line(0,1){.0630252}}
\multiput(34.93,29.93)(.0659341,.032967){13}{\line(1,0){.0659341}}
\multiput(36.644,30.787)(.0659341,.032967){13}{\line(1,0){.0659341}}
\multiput(38.358,31.644)(.0659341,.032967){13}{\line(1,0){.0659341}}
\multiput(40.073,32.501)(.0659341,.032967){13}{\line(1,0){.0659341}}
\multiput(41.787,33.358)(.0659341,.032967){13}{\line(1,0){.0659341}}
\multiput(43.501,34.215)(.0659341,.032967){13}{\line(1,0){.0659341}}
\multiput(45.215,35.073)(.0659341,.032967){13}{\line(1,0){.0659341}}
\multiput(46.93,35.93)(.0659341,.032967){13}{\line(1,0){.0659341}}
\multiput(48.644,36.787)(.0659341,.032967){13}{\line(1,0){.0659341}}
\multiput(50.358,37.644)(.0659341,.032967){13}{\line(1,0){.0659341}}
\multiput(52.073,38.501)(.0659341,.032967){13}{\line(1,0){.0659341}}
\multiput(53.787,39.358)(.0659341,.032967){13}{\line(1,0){.0659341}}
\multiput(55.501,40.215)(.0659341,.032967){13}{\line(1,0){.0659341}}
\multiput(57.215,41.073)(.0659341,.032967){13}{\line(1,0){.0659341}}
\multiput(58.93,41.93)(.0659341,.032967){13}{\line(1,0){.0659341}}
\multiput(60.644,42.787)(.0659341,.032967){13}{\line(1,0){.0659341}}
\multiput(62.358,43.644)(.0659341,.032967){13}{\line(1,0){.0659341}}
\multiput(64.073,44.501)(.0659341,.032967){13}{\line(1,0){.0659341}}
\multiput(65.787,45.358)(.0659341,.032967){13}{\line(1,0){.0659341}}
\multiput(67.501,46.215)(.0659341,.032967){13}{\line(1,0){.0659341}}
\multiput(69.215,47.073)(.0659341,.032967){13}{\line(1,0){.0659341}}
\multiput(-1.07,11.93)(.0659341,.032967){13}{\line(1,0){.0659341}}
\multiput(.644,12.787)(.0659341,.032967){13}{\line(1,0){.0659341}}
\multiput(2.358,13.644)(.0659341,.032967){13}{\line(1,0){.0659341}}
\multiput(4.073,14.501)(.0659341,.032967){13}{\line(1,0){.0659341}}
\multiput(5.787,15.358)(.0659341,.032967){13}{\line(1,0){.0659341}}
\multiput(7.501,16.215)(.0659341,.032967){13}{\line(1,0){.0659341}}
\multiput(9.215,17.073)(.0659341,.032967){13}{\line(1,0){.0659341}}
\multiput(10.93,17.93)(.0659341,.032967){13}{\line(1,0){.0659341}}
\multiput(12.644,18.787)(.0659341,.032967){13}{\line(1,0){.0659341}}
\multiput(14.358,19.644)(.0659341,.032967){13}{\line(1,0){.0659341}}
\multiput(16.073,20.501)(.0659341,.032967){13}{\line(1,0){.0659341}}
\multiput(17.787,21.358)(.0659341,.032967){13}{\line(1,0){.0659341}}
\multiput(19.501,22.215)(.0659341,.032967){13}{\line(1,0){.0659341}}
\multiput(21.215,23.073)(.0659341,.032967){13}{\line(1,0){.0659341}}
\multiput(22.93,23.93)(.0659341,.032967){13}{\line(1,0){.0659341}}
\multiput(24.644,24.787)(.0659341,.032967){13}{\line(1,0){.0659341}}
\multiput(26.358,25.644)(.0659341,.032967){13}{\line(1,0){.0659341}}
\multiput(28.073,26.501)(.0659341,.032967){13}{\line(1,0){.0659341}}
\multiput(29.787,27.358)(.0659341,.032967){13}{\line(1,0){.0659341}}
\multiput(31.501,28.215)(.0659341,.032967){13}{\line(1,0){.0659341}}
\multiput(33.215,29.073)(.0659341,.032967){13}{\line(1,0){.0659341}}
\put(5,5){\makebox(0,0)[cc]{$(0,0)$}}
\put(65,55){\makebox(0,0)[cc]{$(5,4)$}}
\put(15,55){\makebox(0,0)[cc]{$(1,4)$}}
\put(55,5){\makebox(0,0)[cc]{$(4,0)$}}
\end{picture}
\end{center}
\caption{A vertex $p=\{ u_\Box,u_{\kr}\}$ from $P^{**}$,
its four neighbors in $H$,
and some part of $G$.
For the illustrated vertices of $G$,
we use the indicated coordinates,
in particular, 
$u_\Box$ corresponds to $(2,2)$
and 
$u_{\kr}$ corresponds to $(3,2)$.}\label{fig6}
\end{figure}
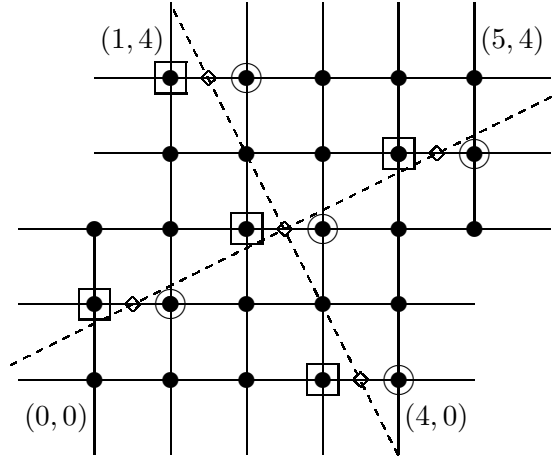
Our next goal is the following.

\begin{lemma}\label{lemma2}
$D^*$ contains an inefficient vertex $u$ 
with ${\rm dist}_G(u,u_\Box)\leq 3$.
\end{lemma}
\begin{proof}
If $p$ has a neighbor in $P(2)$,
the statement is trivial.
Hence, we may assume that no neighbor of $p$ is in $P(2)$.
Hence, by construction, 
$p$ has a neighbor $p'=\{ u'_\Box,u'_{\kr}\}$ from $P({\rm left})$
as well as a neighbor 
$p''=\{ u''_\Box,u''_{\kr}\}$ from $P({\rm right})\cup P(0)$,
in particular, 
$u'_\Box\in D^*$
and 
$u''_\Box\not\in D^*$.
We consider several cases,
where we denote the vertices using coordinates 
as explained in Figure \ref{fig6}.

\medskip

\noindent {\bf Case 1} {\it $(1,4)\in D^*$ and $(4,3)\not\in D^*$.}

\medskip

\noindent $D^*$ contains a vertex $u$ from $N_G[(3,3)]\setminus \{ (4,3)\}$,
which is necessarily inefficient.

\medskip

\noindent {\bf Case 2} {\it $(1,4),(4,3)\in D^*$ and
$(3,0)\not\in D^*$.}

\medskip

\noindent $D^*$ contains a vertex $u$ from $N_G[(3,1)]\setminus \{ (3,0)\}$, which is necessarily inefficient.

\medskip

\noindent {\bf Case 3} {\it $(1,4),(4,3),(3,0)\in D^*$ and
$(0,1)\not\in D^*$.}

\medskip

\noindent $D^*$ contains a vertex $u$ from $N_G[(1,1)]\setminus \{ (0,1)\}$, which is necessarily inefficient.

\medskip

\noindent {\bf Case 4} {\it $(1,4)\not\in D^*$ and $(4,3)\in D^*$.}

\medskip

\noindent $D^*$ contains a vertex from $N_G[(1,3)]$.
In view of the desired results,
we may assume that $D^*$ contains $(0,3)$.
$D^*$ contains a vertex from $N_G[(1,4)]$.
In view of the desired results,
we may assume that $D^*$ contains $(1,5)$.
$D^*$ contains a vertex $u$ from $N_G[(2,4)]$, 
which is now necessarily inefficient.

\medskip

\noindent {\bf Case 5} {\it $(1,4),(4,3)\not\in D^*$ and 
$(3,0)\in D^*$.}

\medskip

\noindent $D^*$ contains a vertex from $N_G[(4,2)]$.
In view of the desired results,
we may assume that $D^*$ contains $(5,2)$.
$D^*$ contains a vertex from $N_G[(4,3)]$.
In view of the desired results,
we may assume that $D^*$ contains $(4,4)$.
$D^*$ contains a vertex $u$ from $N_G[(3,3)]$, 
which is now necessarily inefficient.

\medskip

\noindent {\bf Case 6} {\it $(1,4),(4,3),(3,0)\not\in D^*$ and 
$(0,1)\in D^*$.}

\medskip

\noindent $D^*$ contains a vertex from $N_G[(2,0)]$.
In view of the desired results,
we may assume that $D^*$ contains $(2,-1)$.
$D^*$ contains a vertex $u$ from $N_G[(1,0)]$, 
which is now necessarily inefficient.

\medskip

\noindent The considered cases exhaust all relevant situations,
which completes the proof.
\end{proof}
Since there are 
$13$ vertices at distance at most $2$ 
as well as 
$25$ vertices at distance at most $3$ 
from every vertex of $G$,
Lemma \ref{lemma1} 
and 
Lemma \ref{lemma2}
imply that $D^*$ contains at least
$\frac{k}{20\cdot 25}+O(1)=\frac{k}{500}+O(1)$
inefficient vertices,
and that there is a set $I$ of at least 
$\frac{k}{13\cdot 500}+O(1)=\frac{k}{12500}+O(1)$
inefficient vertices from $D^*$
such that every two vertices in $I$ have pairwise distance at least $3$.
Now, double counting the number $d$ of pairs 
$(u,v)$ with $u\in D^*$ and $v\in N_G[u]$, we obtain 
$n(G)+|I|\leq d=5|D^*|$.
This implies
\begin{eqnarray*}
|D^*| & \geq & \frac{n(G)+|I|}{5}
= k^2+\frac{|I|}{5}
\geq k^2+\frac{k}{62500}+O(1),
\end{eqnarray*}
which completes the proof of Theorem \ref{theorem1}. $\Box$

\section{Proof of Theorem \ref{theorem2}}

Let ${\cal G}$, $\alpha$, $G$, as well as $D$ and $D'$ 
be as in the statement of the theorem.
If $D''$ is a dominating set in $G$,
$D\stackrel{k}{\longleftrightarrow} D''$, and
$D'\stackrel{k}{\longleftrightarrow} D''$,
then 
$D\stackrel{k}{\longleftrightarrow} D'$.
This implies that we may assume that $D'$ 
is a minimum dominating set in $G$.
Since ${\cal G}$ is hereditary,
there is a positive constant $c_2$
such that 
every induced subgraph $G'$ of $G$
has a balanced separator of order $c_2n(G')^\alpha$.
Recursively removing balanced separators,
it follows that there is 
a full binary tree $T$ with root $r$ 
as well as a set $V(t)$ of vertices of $G$ for every vertex $t$ of $T$
such that 
\begin{enumerate}[(i)]
\item $\big(V(t)\big)_{t\in V(T)}$ is a partition of $V(G)$,
\item if $t$ is not a leaf and has depth $d$, 
then $|V(t)|\leq c_2\left(\frac{2}{3}\right)^{\alpha d}n^\alpha$,
\item if $t$ is a leaf, then 
$|V(t)|\leq n^\alpha$, and
\item if $t$ and $t'$ are two distinct vertices of $T$
such that $G$ has an edge between $V(t)$ and $V(t')$,
then 
either $t$ is an ancestor of $t'$ 
or $t'$ is an ancestor of $t$.
\end{enumerate}
Allowing $V(t)$ to be empty, we may assume that all leaves of $T$
have the same depth.
Note that (iii) implies that the depth of $T$ is $\Omega(\log n)$,
and that the important property (iv) 
is a consequence of the definition of (balanced) separators.

For every vertex $t$ of $T$ that is not a leaf, 
we label the two edges between $t$ and its two children
arbitrarily by $0$ (corresponding to left) 
and $1$ (corresponding to right).
For a set $S$ of vertices of $T$ for which $T[S]$ is connected,
let $L_T(S)$ be the set of vertices $t$ of $T$ 
such that there is a path $s_0s_1\ldots s_\ell$ in $T$ with 
$\ell\geq 1$,
$s_0\in S$, 
$s_i\not\in S$ for every $i\in [\ell]$,
$t=s_\ell$,
$s_i$ is a child of $s_{i-1}$ for every $i\in [\ell]$, and 
the edge $s_0s_1$ has label $0$.
Define $R_T(S)$ similarly requiring label $1$ instead of $0$
for the edge $s_0s_1$.
See Figure \ref{fig7}, 
where these definitions are illustrated 
for a root-to-leaf path in $T$.

\begin{figure}[h]
\begin{center}
\unitlength 1mm 
\linethickness{0.4pt}
\ifx\plotpoint\undefined\newsavebox{\plotpoint}\fi 
\begin{picture}(73,119)(0,0)
\put(29,1){\circle*{2}}
\put(34,86){\circle*{2}}
\put(54,106){\circle*{2}}
\put(44,96){\circle*{2}}
\put(64,116){\circle*{2}}
\put(19,11){\circle*{2}}
\put(34,66){\circle*{2}}
\put(9,21){\circle*{2}}
\put(24,76){\circle*{2}}
\put(29,1){\line(-1,1){20}}
\put(44,56){\line(-1,1){20}}
\put(24,76){\line(1,1){20}}
\put(44,96){\line(1,1){20}}
\put(44,56){\circle*{2}}
\put(44,56){\line(-1,-1){35}}
\put(9,21){\line(0,1){0}}
\put(24,76){\line(0,1){0}}
\put(19,11){\line(-1,-1){5}}
\put(34,66){\line(-1,-1){5}}
\put(9,21){\line(-1,-1){5}}
\put(24,76){\line(-1,-1){5}}
\put(44,56){\line(1,-1){5}}
\put(34,86){\line(1,-1){5}}
\put(54,106){\line(1,-1){5}}
\put(44,96){\line(1,-1){5}}
\put(64,116){\line(1,-1){5}}
\put(0,14){\line(0,1){4}}
\put(3,21){\line(1,1){35}}
\put(38,56){\line(-1,1){19}}
\multiput(19,75)(-.0336879433,-.0833333333){564}{\line(0,-1){.0833333333}}
\put(0,28){\line(0,-1){11}}
\put(64,119){\makebox(0,0)[cc]{$r=t_0$}}
\put(54,109){\makebox(0,0)[cc]{$t_1$}}
\put(44,99){\makebox(0,0)[cc]{$t_2$}}
\put(34,89){\makebox(0,0)[cc]{$t_3$}}
\put(24,79){\makebox(0,0)[cc]{$t_4$}}
\put(34,69){\makebox(0,0)[cc]{$t_5$}}
\put(44,59){\makebox(0,0)[cc]{$t_6$}}
\put(13,21){\makebox(0,0)[cc]{$t_7$}}
\put(21,13){\makebox(0,0)[cc]{$t_8$}}
\put(31,3){\makebox(0,0)[cc]{$t_9$}}
\put(35,82){\line(1,1){34}}
\put(32,103){\makebox(0,0)[cc]{$P$}}
\put(58,83){\makebox(0,0)[cc]{$R_T(V(P))$}}
\put(21,50){\makebox(0,0)[cc]{$L_T(V(P))$}}
{\tiny
\put(28,82){\makebox(0,0)[cc]{$0$}}
\put(38,92){\makebox(0,0)[cc]{$0$}}
\put(48,102){\makebox(0,0)[cc]{$0$}}
\put(58,112){\makebox(0,0)[cc]{$0$}}
\put(25,39){\makebox(0,0)[cc]{$0$}}
\put(40,63){\makebox(0,0)[cc]{$1$}}
\put(30,73){\makebox(0,0)[cc]{$1$}}
\put(25,7){\makebox(0,0)[cc]{$1$}}
\put(15,17){\makebox(0,0)[cc]{$1$}}
}
\put(69,116){\line(1,-1){4}}
\put(35,82){\line(1,-1){20}}
\put(55,62){\line(-1,-1){12}}
\put(31,-4){\line(0,1){0}}
\put(3,21){\line(1,-1){16}}
\put(0,14){\line(1,-1){14}}
\put(14,0){\line(1,1){5}}
\put(43,50){\line(1,-1){6}}
\put(49,44){\line(1,1){17}}
\multiput(73,112)(-.033653846,-.245192308){208}{\line(0,-1){.245192308}}
\end{picture}
\end{center}
\caption{A root-to-leaf path $P:t_0t_1\ldots t_9$ and 
the two sets $L_T(V(P))$ and $R_T(V(P))$.}
\label{fig7}
\end{figure}
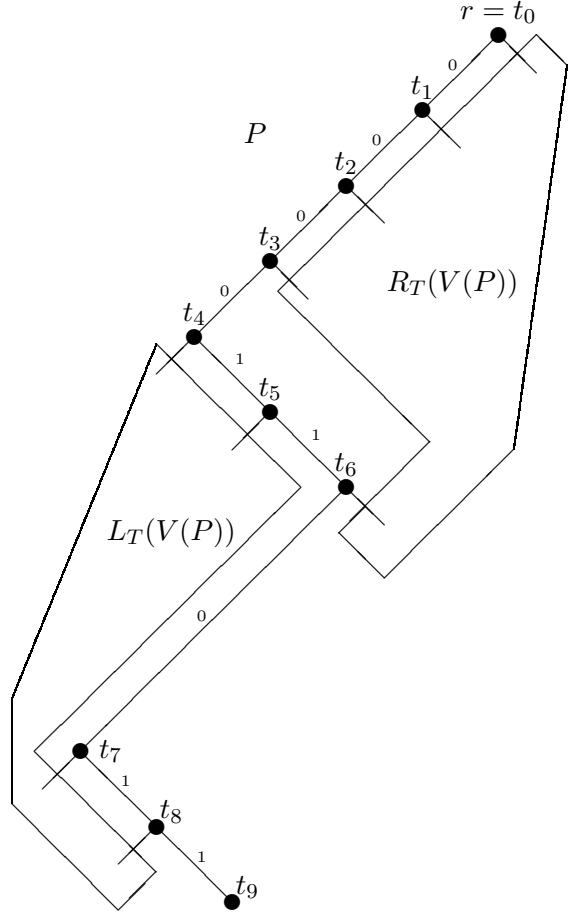
Let 
$$
V_G(S)=\bigcup_{t\in S}V(t),\,\,\,\,\,\,\,\,\,\,\,
L_G(S)=\bigcup_{t\in L_T(S)}V(t),
\,\,\,\,\,\,\,\,\,\,\,\mbox{ and }\,\,\,\,\,\,\,\,\,\,\,
R_G(S)=\bigcup_{t\in R_T(S)}V(t).$$
Note that, if $S$ contains the root of $T$, then 
\begin{itemize}
\item the sets $S$, $L_T(S)$, and $R_T(S)$ partition $V(T)$,
and 
\item the sets $V_G(S)$, 
$L_G(S)$,
and 
$R_G(S)$ partition $V(G)$,
and there is no edge between 
$L_G(S)$ and $R_G(S)$.
\end{itemize}
If $P:t_0t_1\ldots t_d$ is a root-to-leaf path in $T$ 
with $r=t_0$, 
then the above properties imply that the set
$$D(P)=
V_G(V(P))\cup \Big(L_G(V(P))\cap D'\Big)
\cup \Big(R_G(V(P))\cap D\Big)$$
is a dominating set in $G$.
Note, for instance, that all neighbors of vertices from 
$L_G(V(P))$
outside of 
$L_G(V(P))$
lie in 
$V_G(V(P))$.
We call $D(P)$ the {\it special dominating set associated to $P$}.
Note that 
\begin{eqnarray}
\left|V_G(V(P))\right| & \stackrel{(ii),(iii)}{\leq} &
\sum_{i\in [\ell-1]_0}c_2\left(\frac{2}{3}\right)^{\alpha i}n^\alpha
+n^\alpha\nonumber\\
&\leq &\left(\frac{c_2}{1-\left(\frac{2}{3}\right)^{\alpha}}+1\right)n^\alpha
=c_3n^\alpha,
\,\,\,\,\,\,\,\,\mbox{ where $c_3=\left(\frac{c_2}{1-\left(\frac{2}{3}\right)^{\alpha}}+1\right)$.}\label{epath}
\end{eqnarray}
The following lemma contains the important observation 
that $D(P)$ is not much larger than $D$.

\begin{lemma}\label{lemma3}
$|D(P)|\leq |D|+2c_3n^\alpha$.
\end{lemma}
\begin{proof}
Suppose, for a contradiction, that 
$|D(P)|>|D|+2c_3n^\alpha$.

By the definition of $D(P)$, we obtain 
\begin{eqnarray*}
c_3n^\alpha+\Big|L_G(V(P))\cap D'\Big|
&+&\Big|R_G(V(P))\cap D\Big|\\
& \geq & |D(P)|\\
& >& |D|+2c_3n^\alpha\\
& \geq & \Big|L_G(V(P))\cap D\Big|
+\Big|R_G(V(P))\cap D\Big|+2c_3n^\alpha,
\end{eqnarray*}
and, hence, 
$$\Big|L_G(V(P))\cap D'\Big|
>\Big|L_G(V(P))\cap D\Big|+c_3n^\alpha.$$
Now,
$$D''=V_G(V(P))
\cup \Big(L_G(V(P))\cap D\Big)
\cup \Big(R_G(V(P))\cap D'\Big)$$
is a dominating set in $G$ with 
\begin{eqnarray*}
|D''| &\leq &
c_3n^\alpha+\Big|L_G(V(P))\cap D\Big|
+\Big|R_G(V(P))\cap D'\Big|\\
& < & \Big|L_G(V(P))\cap D'\Big|
+\Big|R_G(V(P))\cap D'\Big|\\
& \leq & |D'|,
\end{eqnarray*}
contradicting the fact that $D'$ is a minimum dominating set.
This completes the proof.
\end{proof}
In order to prove the theorem,
we will construct a sequence 
$D_0,D_1,\ldots,D_\ell$ 
of dominating sets in $G$ such that 
$D=D_0$, 
$D'=D_\ell$, 
$D_{i-1}$ is adjacent to $D_i$ for every $i\in [\ell]$, and, 
for every $i\in [\ell]_0$, there is some $j\in [\ell]_0$ with
\begin{itemize}
\item $|j-i|\leq 2c_3n^\alpha$ and
\item $D_j$ is a special dominating set associated to some
root-to-leaf path in $T$. 
\end{itemize}
Together with Lemma \ref{lemma3},
this implies
$$|D_i|\leq |D|+4c_3n^\alpha\mbox{ for every $i\in [\ell]_0$},$$
which completes the proof.

For an edge $e$ of $T$, let $\sigma(e)\in \{ 0,1\}$ be its label.
For a root-to-leaf path $P:t_0t_1\ldots t_d$ in $T$ with $r=t_0$,
let
$$\sigma(P)=\Big(\sigma(t_0t_1),
\sigma(t_1t_2),\ldots,\sigma(t_{d-1}t_d)\Big).$$
We consider the root-to-leaf paths $P$ 
according to the lexicographic order of the $\sigma(P)$'s,
that is, $P$ comes immediately before $P'$
if there is some $i\in [d-1]$ such that
\begin{eqnarray*}
\sigma(P)&=&\Big(\sigma(t_0t_1),\ldots,\sigma(t_{i-1}t_i),0,1,\ldots,1\Big)
\mbox{ and }\\
\sigma(P')&=&\Big(\sigma(t_0t_1),\ldots,\sigma(t_{i-1}t_i),1,0,\ldots,0\Big).
\end{eqnarray*}
See Figure \ref{fig8} for an illustration.

\begin{figure}[h]
\begin{center}
\unitlength 1mm 
\linethickness{0.4pt}
\ifx\plotpoint\undefined\newsavebox{\plotpoint}\fi 
\begin{picture}(88,119)(0,0)
\put(29,1){\circle*{2}}
\put(59,1){\circle*{2}}
\put(34,86){\circle*{2}}
\put(54,106){\circle*{2}}
\put(44,96){\circle*{2}}
\put(64,116){\circle*{2}}
\put(19,11){\circle*{2}}
\put(69,11){\circle*{2}}
\put(34,66){\circle*{2}}
\put(9,21){\circle*{2}}
\put(79,21){\circle*{2}}
\put(24,76){\circle*{2}}
\put(29,1){\line(-1,1){20}}
\put(59,1){\line(1,1){20}}
\put(44,56){\line(-1,1){20}}
\put(24,76){\line(1,1){20}}
\put(44,96){\line(1,1){20}}
\put(44,56){\circle*{2}}
\put(44,56){\line(-1,-1){35}}
\put(44,56){\line(1,-1){35}}
\put(9,21){\line(0,1){0}}
\put(79,21){\line(0,1){0}}
\put(24,76){\line(0,1){0}}
\put(19,11){\line(-1,-1){5}}
\put(69,11){\line(1,-1){5}}
\put(34,66){\line(-1,-1){5}}
\put(9,21){\line(-1,-1){5}}
\put(79,21){\line(1,-1){5}}
\put(24,76){\line(-1,-1){5}}
\put(44,56){\line(1,-1){5}}
\put(34,86){\line(1,-1){5}}
\put(54,106){\line(1,-1){5}}
\put(44,96){\line(1,-1){5}}
\put(64,116){\line(1,-1){5}}
\put(0,14){\line(0,1){4}}
\put(3,21){\line(1,1){35}}
\put(38,56){\line(-1,1){19}}
\multiput(19,75)(-.0336879433,-.0833333333){564}{\line(0,-1){.0833333333}}
\put(0,28){\line(0,-1){11}}
\put(88,28){\line(0,-1){11}}
\put(64,119){\makebox(0,0)[cc]{$r=t_0$}}
\put(54,109){\makebox(0,0)[cc]{$t_1$}}
\put(44,99){\makebox(0,0)[cc]{$t_2$}}
\put(34,89){\makebox(0,0)[cc]{$t_3$}}
\put(24,79){\makebox(0,0)[cc]{$t_4$}}
\put(34,69){\makebox(0,0)[cc]{$t_5$}}
\put(44,59){\makebox(0,0)[cc]{$t_6$}}
\put(13,21){\makebox(0,0)[cc]{$t_7$}}
\put(75,21){\makebox(0,0)[cc]{$t'_7$}}
\put(21,13){\makebox(0,0)[cc]{$t_8$}}
\put(66,13){\makebox(0,0)[cc]{$t'_8$}}
\put(31,3){\makebox(0,0)[cc]{$t_9$}}
\put(56,3){\makebox(0,0)[cc]{$t'_9$}}
\put(35,82){\line(1,1){34}}
\put(58,83){\makebox(0,0)[cc]{$R_T(V(P'))$}}
\put(21,50){\makebox(0,0)[cc]{$L_T(V(P))$}}
{\tiny
\put(28,82){\makebox(0,0)[cc]{$0$}}
\put(38,92){\makebox(0,0)[cc]{$0$}}
\put(48,102){\makebox(0,0)[cc]{$0$}}
\put(58,112){\makebox(0,0)[cc]{$0$}}
\put(25,39){\makebox(0,0)[cc]{$0$}}
\put(63,39){\makebox(0,0)[cc]{$1$}}
\put(40,63){\makebox(0,0)[cc]{$1$}}
\put(30,73){\makebox(0,0)[cc]{$1$}}
\put(25,7){\makebox(0,0)[cc]{$1$}}
\put(63,7){\makebox(0,0)[cc]{$0$}}
\put(15,17){\makebox(0,0)[cc]{$1$}}
\put(73,17){\makebox(0,0)[cc]{$0$}}
}
\put(69,116){\line(1,-1){4}}
\put(31,-4){\line(0,1){0}}
\put(3,21){\line(1,-1){16}}
\put(85,21){\line(-1,-1){16}}
\put(0,14){\line(1,-1){14}}
\put(88,14){\line(-1,-1){14}}
\put(14,0){\line(1,1){5}}
\put(74,0){\line(-1,1){5}}
\multiput(73,112)(.0337078652,-.1887640449){445}{\line(0,-1){.1887640449}}
\multiput(35,82)(.03373819163,-.04116059379){1482}{\line(0,-1){.04116059379}}
\put(88,17){\line(0,-1){3}}
\end{picture}
\end{center}
\caption{The two paths $P:t_0t_1\ldots t_6t_7t_8t_9$ and 
$P':t_0t_1\ldots t_6t_7't_8't_9'$ are lexicographically consecutive.}
\label{fig8}
\end{figure}
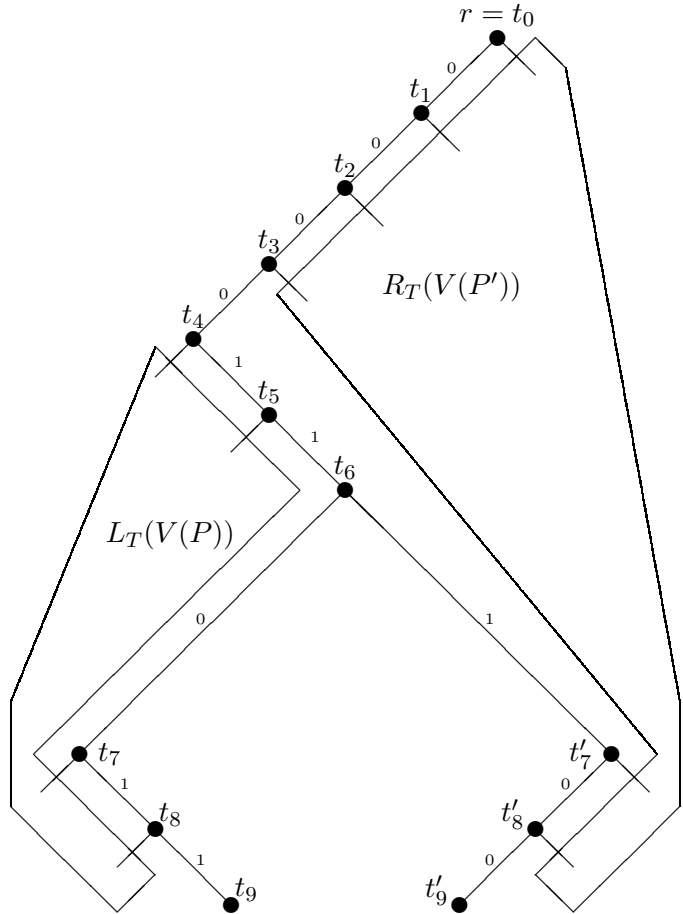
Let 
$$P(0,\ldots,0,0),P(0,\ldots,0,1),P(0,\ldots,1,0),P(0,\ldots,1,1),\ldots,
P(1,\ldots,1,0),
P(1,\ldots,1,1)$$
be this order
of the root-to-leaf paths in $T$.

Now, we describe how to construct the sequence $D_0,\ldots,D_\ell$
mentioned after Lemma \ref{lemma3}.
Starting with $D_0=D$, we first transform $D$ into $D(P(0,\ldots,0,0))$,
where $P(0,\ldots,0,0):t_0t_1\ldots t_d$ as follows:

\begin{algorithm}[H]
\For{$i=0$ \KwTo $d$}
{
Add the vertices from $V(t_i)\setminus D$ 
one by one to the current dominating set\;
}
\end{algorithm}
Since 
$\left|V_G(V(P))\right|\stackrel{(\ref{epath})}{\leq} c_3n^\alpha$,
this requires at most $c_3n^\alpha$ steps.
Symmetrically, the set $D(P(1,\ldots,1,1))$
can be transformed into $D'$
by a sequence of at most $c_3n^\alpha$ steps.
In order to complete the proof,
we explain how to transform $D(P)$ into $D(P')$,
where 
$P:t_0t_1\ldots t_it_{i+1}\ldots t_d$,
$P':t_0t_1\ldots t_it'_{i+1}\ldots t'_d$,
and $P$ comes immediately before $P'$ 
in the lexicographic order:

\begin{algorithm}[H]
\For{$j=d$ \KwTo $i+1$}
{
Remove from the current dominating set
one by one 
the vertices from $V(t_j)\setminus D'$\;
}
\For{$j=i+1$ \KwTo $d$}
{
Add the vertices from $V(t'_j)\setminus D$
one by one to the current dominating set\;
}
\end{algorithm}
Since 
$$\left|\bigcup_{j\in \{ i+1,\ldots,d\}}V(t_j)\right|+
\left|\bigcup_{j\in \{ i+1,\ldots,d\}}V(t_j')\right|
\leq 
\left|V_G(V(P))\right|+\left|V_G(V(P'))\right|
\stackrel{(\ref{epath})}{\leq} 2c_3n^\alpha,$$
this requires at most $2c_3n^\alpha$ steps.
It is easy to see that all intermediate sets 
during these transformations are dominating 
sets in $G$, which completes the proof
of Theorem \ref{theorem2}. $\Box$

\end{document}